\documentclass[11pt]{amsart} \usepackage{latexsym, amssymb, stmaryrd,amsthm, amsmath, amscd}
\usepackage{mathscinet}

\newif\iffull
\fullfalse

\makeatletter
\@namedef{subjclassname@2010}{%
  \textup{2010} Mathematics Subject Classification}
\makeatother

\usepackage{hyperref}

\newtheorem{thm}{Theorem}[section]
\newtheorem{lemma}[thm]{Lemma}
\newtheorem{prop}[thm]{Proposition}
\newtheorem{cor}[thm]{Corollary}

\theoremstyle{definition}
\newtheorem{df}[thm]{Definition}

\newtheorem{ex}[thm]{Example}
\newtheorem{conv}[thm]{Convention}
\newtheorem{rmk}[thm]{Remark}

\newtheorem{claim}{Claim}

 \newenvironment{claimproof}{\begin{proof}}{\end{proof}}

\renewcommand{\r}{\mathbb{R}}
\newcommand{\Z}{\mathbb{Z}}
\newcommand{\setmid}{ :}

\newcommand{\curly}[1]{\mathcal{#1}}

\newcommand{\B}{\curly{B}}

\newcommand{\M}{\mathfrak{M}}
\newcommand{\N}{\mathfrak{N}}
\newcommand{\n}{\mathbb{N}}
\newcommand{\cU}{\curly{U}}

\newcommand{\cN}{\mathcal{N}}
\newcommand{\cM}{\mathcal{M}}

\newcommand{\la}{\curly{L}}

\renewcommand{\to}{\rightarrow}

\def \T{\operatorname{\mathcal{T}}}

\def \q {{\mathbb Q}}
\def \<{\langle}
\def \>{\rangle}

\def \*Z {{{^*}\Z}}
\def \((  {(\!(}
\def \)) {)\!)}

\def \tp{\operatorname{tp}}

\numberwithin{equation}{section}

\def \U{\mathcal{U}}

\def \AML{\operatorname{AML}}
\def \dfbl{\operatorname{Def}}
\def \rk{\operatorname{rk}}
\def \G{\mathfrak{G}}

\makeatletter

\def\dotminussym#1#2{%
  \setbox0=\hbox{$\m@th#1-$}%
  \kern.5\wd0%
  \hbox to 0pt{\hss\hbox{$\m@th#1-$}\hss}%
  \raise.6\ht0\hbox to 0pt{\hss$\m@th#1.$\hss}%
  \kern.5\wd0}

\allowdisplaybreaks[2]

\subjclass[2010]{Primary 03B48; Secondary 03C20, 60A10}

\keywords{first-order logic, ultraproducts, measure theory}

\begin{document}

\title{An Approximate Logic for Measures}
\author{Isaac Goldbring and Henry Towsner}
\thanks{Goldbring's work was partially supported by NSF grant DMS-1007144 and Towsner's work was partially supported by NSF grant DMS-1001528.}
\address {University of Illinois at Chicago, Department of Mathematics, Statistics, and Computer Science, Science and Engineering Offices (M/C 249), 851 S. Morgan St., Chicago, IL 60607-7045, USA}
\email{isaac@math.uic.edu}
\urladdr{www.math.uic.edu/~isaac}

\address {University of Pennsylvania, Department of Mathematics, 209 S. 33rd Street, Philadelphia, PA 19104-6395}
\email{htowsner@math.upenn.edu}
\urladdr{www.sas.upenn.edu/~htowsner}

\begin{abstract}
We present a logical framework for formalizing connections between finitary combinatorics and measure theory or ergodic theory that have appeared in various places throughout the literature.  We develop the basic syntax and semantics of this logic and give applications, showing that the method can express the classic Furstenberg correspondence and to short proofs of the Szemer\'edi Regularity Lemma and the hypergraph removal lemma.  We also derive some connections between the model-theoretic notion of stability and the Gowers uniformity norms from combinatorics.
\end{abstract}

\maketitle

\section{Introduction}

Since the 1970's, diagonalization arguments (and their generalization to ultrafilters) have been used to connect results in finitary combinatorics with results in measure theory or ergodic theory \cite{furstenberg77}.  Although it has long been known that this connection can be described with first-order logic, a recent striking paper by Hrushovski \cite{hrushovski} demonstrated that modern developments in model theory can be used to give new substantive results in this area as well.

Papers on various aspects of this interaction have been published from a variety of perspectives, with almost as wide a variety of terminology \cite{tao07,elek07,elek08,austin:MR2747063,towsner:MR2529651,furstenberg:MR1039473,MR891243}.  Our goal in this paper is to present an assortment of these techniques in a common framework.  We hope this will make the entire area more accessible.

We are typically concerned with the situation where we wish to prove a statement 1) about sufficiently large finite structures which 2) concerns the finite analogs of measure-theoretic notions such as density or integrals.  The classic example of such a statement is Szemer\'edi's Theorem (other examples can be found later in the paper and in the references):
\begin{thm}[Szemer\'edi's Theorem]
For each $\epsilon>0$ and each $k\in \n^{>0}$, there is an $N$ such that whenever $n>N$ and $A\subseteq [1,n]:=\{1,\ldots,n\}$ with $\frac{|A|}{n}\geq\epsilon$, there exist $a,d$ such that $a,a+d,a+2d,\ldots,a+(k-1)d\in A$.
\end{thm}
Suppose this statement were false, and therefore that for some particular $\epsilon$, some $k$, and infinitely many $n$, there were sets $A_n\subseteq[1,n]$ serving as counter-examples to the theorem. For each $n$, we could view the set $[1,n]$, together with the subset $A_n$ and the function $+\mod n$, as a \emph{structure} $\M_n$ of first-order logic.  Using the ultraproduct construction, these finite models can be assembled into a single \emph{infinite} model $\mathfrak{M}$ satisfying the \L o\'s Theorem, which, phrased informally, states:
\begin{thm} 
A first-order sentence is true in $\mathfrak{M}$ if and only the sentence is true in $\mathfrak{M}_n$ for ``almost every'' $n$.
\end{thm}
The existence of $a$ and $d$ satisfying the theorem can be expressed by a first-order sentence, so if we could show that this sentence was true in the infinite structure $\M$, then the same sentence would be true in many of the finite structures $\M_n$, which would contradict our assumption.

In each of the finite models, we have a natural measure on $[1,n]$, namely the (normalized) counting measure $\mu_n$, and the assumption that $\frac{|A_n|}{n}\geq\epsilon$ is simply the statement that $\mu_n(A_n)\geq\epsilon$.  We would like to produce a measure on $\M$ which satisfies some analog of the \L o\'s theorem.

One way to accomplish this is to expand the language of our models with many new predicates for measures, one for each definable set $S$ and rational $\delta$, and specify that the formula $m_{S,\delta}$ holds in $\M_n$ exactly if $\mu_n(S)<\delta$.  Then, for instance, the failure of $m_{A,\epsilon}$ in the finite structures $\M_n$ implies the failure of the same formula in $\M$.  With a bit of work, these new predicates are enough to ensure that we can construct a genuine measure, the Loeb measure (see \cite{goldblatt}), on the structure $\M$ whose properties are related to those of the finite structures via the \L o\'s theorem.  The extension of a language of first-order logic by these predicates has been used in \cite{hrushovski}, but the details of the construction have not appeared in print.

The applications of these predicates for measures are sufficiently compelling to justify the study of this extension of first-order logic in its own right.  It would be quite interesting to see certain results about first-order logic proven for this extension; for example, some of the results in \cite{pillay:MR1650667} can be viewed as studying the addition of a ``random'' predicate $P$ to a language of first-order logic, a notion which has natural application in the presence of measures if the result can be extended to include $\AML$.  

In this paper we define \emph{approximate measure logic}, or \emph{AML}, to be an extension of first-order logic in which these new predicates are always present.  This serves two purposes.  First, to give the full details of extending first-order logic by measure predicates in a fairly general way.  Second, we hope that providing a precise formulation will spur the investigation of the distinctive properties of AML.

In general, a logic which intends to talk about measure has to make some kind of compromise, since there is a tension between talking about measures in a natural way and talking about the usual first-order notions (which, for instance, impose the presence of projections of sets).  AML is not the first attempt at such a combination (see \cite{keisler:MR819545,MR2505436}), but differs because it stays much closer to conventional first-order logic.  The main oddity, relative to what a naive attempt at a logic for measures might look like, is that the description of measures in the logic is \emph{approximate}: there can be a slight mismatch between which sentences are true in a structure and what is true about the actual measure.  For instance, it is possible to have a structure $\M$ together with a measure $\mu$ so that $\M$ satisfies a sentence saying the measure of a set $B$ is strictly less than $\epsilon$ while in fact the measure of $B$ is precisely $\epsilon$.  This is necessary to accommodate the ultraproduct construction.

The next three sections lay out the basic definitions of AML and describe some properties of structures of AML, culminating in the construction of ultraproducts which have well-behaved measures.  In Section \ref{sec:correspondence}, we present two examples of simple applications, giving a proof of the Furstenberg correspondence between finite sets and dynamical systems, and a proof of the Szemer\'edi Regularity Lemma.

In Section \ref{sec:first_order_translation} we prove a Downward L\"owenheim-Skolem theorem for $\AML$.  In Section \ref{sec:applications}, we illustrate several useful model-theoretic techniques in AML and apply these to show a connection between definability, the model-theoretic notion of stability, and combinatorial quantities known as the Gowers uniformity norms.  We intend this section to be a survey of methods that have proven useful.  These methods are either more fully developed in the references, or await a fuller development.

In this article, $m$ and $n$ range over $\n:=\{0,1,2,\ldots\}$.  For a relation $R\subseteq X\times Y$, $x\in X$, and $y\in Y$, we set $R^x:=\{y\in Y \setmid (x,y)\in R\}$ and $R_y:=\{x\in X\setmid (x,y)\in R\}$.

The authors thank Ehud Hrushovski, Lou van den Dries, and the members of the UCLA reading seminar on approximate groups:  Matthias Aschenbrenner, Greg Hjorth, Terence Tao, and Anush Tserunyan,  for helpful conversations on the topic of this paper.

\section{Syntax and Semantics}
\label{sec:syntax_semantics}

A \emph{signature} for \emph{approximate measure logic} $(\AML)$ is nothing more than a first-order signature.  Suppose that $\la$ is a first-order signature.  One forms $\AML$ $\la$-terms as usual, by applying function symbols to constant symbols and variables.  We let $\T$ denote the set of $\la$-terms.

\begin{df}
We define the \emph{$\la$-formulae} of $\AML$ as follows:
\begin{itemize}
\item Any first-order atomic $\la$-formula is an $\AML$ $\la$-formula.
\item The $\AML$ $\la$-formulae are closed under the use of the usual connectives $\neg$ and $\wedge$.
\item If $\varphi$ is an $\AML$ $\la$-formula and $x$ is a variable, then $\exists x\varphi$ and $\forall x\varphi$ are $\AML$ $\la$-formulae.
\item If $\vec{x}=(x_1,\ldots,x_n)$ is a sequence of distinct variables, $q$ is a non-negative rational, $\bowtie \in\{<,\leq\}$, and $\varphi$ is an $\la$-formulae, then $m_{\vec{x}}\bowtie q.\varphi$ is an $\la$-formula.  
\end{itemize}
\end{df}

The last constructor, which we will refer to as the \emph{measure constructor}, in the above definition behaves like a quantifier in the sense that any of the variables $x_1,\ldots,x_n$ appearing free in $\varphi$ become bound in the resulting formula.

From now on, $\bowtie$ will always denote $<$ or $\leq$.
\begin{conv}
It will become convenient to introduce the following notations:
\begin{itemize}
\item $m_{\vec{x}}\geq q.\varphi$ is an abbreviation for $\neg m_{\vec{x}}<q.\varphi$
\item $m_{\vec{x}}>q.\varphi$ is an abbreviation for $\neg m_{\vec{x}}\leq q.\varphi$
\end{itemize}
\end{conv}

\begin{rmk} 
It is not necessary that $\AML$ contain precisely one measure for each tuple length.  A more general definition could allow multiple measures at each arity.  In this case, one would also want to track which measures on pairs, for instance, correspond to the products of which measures on singletons, and so on.  There are even situations where it is natural to have a measure on pairs which does not correspond to any product of measures on singletons.\footnote{A typical example is the proofs in extremal graph theory of sparse analogs of density statements, as in the regularity lemma for sparse graphs \cite{kohayakawa:MR1661982}, where the counting measure on pairs is replaced by counting the proportion of pairs from some fixed (typically small) set of pairs.  If the fixed set of pairs satisfies certain randomness properties, this measure cannot be viewed as any product measure.}  Further generalizations could involve such things as measures valued in ordered rings.  A fully general definition of $\AML$ would include in the signature $\la$, for each arity, a list of measures together with the ordered ring in which they are valued.  In this paper we eschew these complications and focus only on the simplest case, with a single real-valued measure; the changes needed to handle more general cases are largely routine.
\end{rmk}

Particular examples of $\AML$ structures are given by finite structures with the normalized counting measure and by structures with definable Keisler measures.  However we would like to define an \emph{a priori} notion of semantics for $\AML$: analogously to the definition of a structure of first-order logic, we would like to first state that a certain object is an $\AML$ structure and then interpret formulas and sentences within that object.  The heart of such a definition should be that an $\AML$ structure consists of a first-order structure together with a measure such that all definable sets are measurable.  Since identifying the definable sets requires interpreting the logic, this makes it difficult to recognize a valid $\AML$ structure without simultaneously interpreting the logic.

We adopt a compromise: an \emph{$\AML$ quasistructure} is an object which would be an $\AML$ structure except that definable sets may not be measurable.  We define the rank of a formula to be the level of nesting of measure quantifiers, and we then, by simultaneous recursion, identify the definable sets of rank $n$ and the quasistructures in which all sets defined by formulas of rank $n$ are measurable; as long as all sets defined by formulas of rank $n$ are measurable, it will be possible to interpret formulas of rank $n+1$.  An \emph{$\AML$ structure} will then be a quasistructure in which this recursion continues through all ranks.

We face one additional obstacle: a first-order structure together with a measure does not provide enough information to completely specify the semantics.  In order to satisfy compactness, we need to allow the possibility that 
\[\M\vDash m_{\vec x}<r+\epsilon.\varphi(\vec x)\]
for all $\epsilon>0$, but 
\[\M\not\vDash m_{\vec x}\leq r.\varphi(\vec x).\]
In such a structure, the measure of the set defined by $\varphi$ should be $r$, but this structure must have a different semantics from a structure $\N$ where the measure of the set defined by $\varphi$ is $r$ and
\[\N\vDash m_{\vec x}\leq r.\varphi(\vec x).\]
To accomodate this, we pair the measure portion of a structure with a function $v$ which assigns to each set $X$ a value $\odot$, $\oplus$, or $\ominus$.  The values $\oplus$ and $\ominus$ indicate that the structure can only identify the measure of the set $X$ as it is approximated from above or from below, respectively, while the value $\odot$ indicates that the structure actually identifies the measure of the set $X$ exactly.  This issue can only occur when the value of the measure is rational, so if the measure of $X$ is irrational, we always set $v(X)=\odot$.

\begin{df}
Let $\mathcal L$ be a first-order signature.  An \emph{$\AML$ $\la$-quasistructure} consists of:
\begin{itemize}
\item a first-order $\la$-structure $\M$;
\item for each $n\geq 1$, an algebra $\B^\M_n$ of subsets of $M^n$ and a finitely additive measure $\mu^\M_n$ on $\B^\M_n$ such that, for any $m,n$, $\B^\M_{m+n}$ extends $\B^\M_m\otimes\B^\M_n$ and $\mu^\M_{m+n}$ extends $\mu^\M_m\times\mu^\M_n$;
\item for each set $X\in \B^\M_n$, an assignment of a value $v^\M_n(X)\in \{\oplus,\ominus,\odot\}$ such that $v^\M_n(X)=\odot$ whenever $\mu^\M_n(X)\notin \q$.
\end{itemize}
\end{df}

If the structure $\M$ and the arity $n$ are clear from context, we will write $\mu$ and $v$ instead of $\mu^\M_n$ and $v^\M_n$.

Suppose that $\M$ is an $\AML$ $\la$-quasistructure.  Let $V$ be the set of variables and suppose also that $s:V\to M$ is a function; we call such an $s$ a \emph{valuation on $\M$}.  
Given any sequence of distinct variables $\vec{v}=(v_1,\ldots,v_n)$ and any sequence $\vec{c}=(c_1,\ldots,c_n)$ from $M$, we let $s(\vec{c}/\vec{v})$ be the valuation obtained from $s$ by redefining (if necessary) $s$ at $v_i$ to take the value $c_i$.  

We would now like to define the satisfaction relation $\M\models \varphi[s]$ by recursion on complexity of formulae in such a way that this extends the satisfaction relation for first-order logic.  However, in order to specify the semantics of the measure constructor, one needs to know that the set defined by the formula following the measure constructor is measurable.  Thus, we need to define the semantics in stages while ensuring that the sets defined at each stage are measurable.

\
Towards this end, we define the \emph{rank} of an $\AML$ $\la$-formula $\varphi$, denoted $\rk(\varphi)$, by induction:  
\begin{itemize}
\item If $\varphi$ is a classical $\la$-formula, then $\rk(\varphi)= 0$; 
\item $\rk(\neg \varphi)=\rk(\forall x\varphi)=\rk(\varphi)$;
\item $\rk(\varphi\vee \psi)=\max(\rk(\varphi),\rk(\psi))$;
\item $\rk(m_{\vec{x}}\bowtie r.\varphi(\vec{x},\vec{y}))=\rk(\varphi)+1$.
\end{itemize}

Suppose that the satisfaction relation $\M\models \varphi(\vec x)[s]$ has been defined for formulae of rank $\leq n$.  (Note that this is automatic for $n=0$.)  As is customary, we often write $\M\models \varphi(\vec{a})$ if $\varphi(\vec x)$ is a formula with free variables among $\vec x$, $s$ is a valuation for which $s(x_i)=a_i$, and $\M\models \varphi[s]$.  If $\varphi(\vec x,\vec y)$ is a formula of rank $\leq n$ and $b\in M^{|\vec y|}$, we then set $\varphi(M,\vec b):=\{a\in M^{|\vec x|} \ : \ \M\models \varphi(\vec a, \vec b)\}$.  We also use the notation $v(\varphi,\vec b)$ to denote $v(\varphi(M,\vec b))$.

We say that the $\AML$ $\la$-quasistructure $\M$ satisfies the \emph{rank $n$ measurability condition} if, for all formulae $\varphi(\vec x,\vec y)$ of rank $\leq n$ and all $b\in M^{|\vec y|}$, we have $\varphi(M,b)\in \mathcal B^{\mathfrak{M}}_{|\vec x|}$.  

Assume that $\M$ satisfies the rank $n$ measurability condition.  We now define the satisfaction relation $\models$ for rank $n+1$ formulae.  We treat the connective and quantifier cases as in first-order logic and only specify how to deal with the semantics of the measure constructor.  Suppose that $\varphi(\vec x,\vec y)$ is a formula of rank $n$ and $s$ is a valuation such that $s(\vec y)=\vec b$.  We then declare:

\begin{itemize}
\item $\M\models m_{\vec x}<r.\varphi(\vec x, \vec b)$ if and only if:
\begin{itemize}
\item $\mu(\varphi(M, \vec b))<r$, or
\item $\mu(\varphi(M, \vec b))=r$ and $v(\varphi,\vec b)=\ominus$.
\end{itemize}
\item $\M\models m_{\vec x}\leq r.\varphi(\vec x;\vec b)$ if and only if:
\begin{itemize}
\item $\mu(\varphi(M, \vec b))<r$, or
\item $\mu(\varphi(M, \vec b))=r$ and $v(\varphi,\vec b)\not=\oplus$.
\end{itemize}
\end{itemize}

\begin{df}
We say that the $\AML$ quasistructure $\M$ is a \emph{structure} if it satisfies the rank $n$ measurability condition for all $n$.
\end{df}

If $\varphi$ has no free variables, then $\varphi$ is called a \emph{sentence}.  If $T$ is a set of sentences and $\M$ is a structure, we write $\M\models T$ if and only if $\M\models \varphi$ for all $\varphi\in T$; in this case, we say that $\M$ is a \emph{model} of $T$.

\begin{ex}
Let $\la$ be the signature consisting of a single binary function symbol $\cdot$ and a single constant symbol $e$.  Let $G$ be a group with $|G|=m$.  We view $G$ as a first-order $\la$-structure $\mathfrak{G}$ by interpreting $\cdot$ as multiplication in $G$ and $e$ as the identity of $G$.  In order to view $\mathfrak{G}$ as an $\AML$ $\la$-quasistructure, we set $\B^{\mathfrak{G}}_n:=\mathcal P(G^n)$ and $\mu_n(X):=\frac{|X|}{m^n}$, that is, we equip each cartesian power $G^n$ with the (normalized) counting measure.  It is clear that the $\AML$ quasistructure $\mathfrak G$ is actually a structure.

Now fix $g\in G$ and let $\varphi(x,y)$ be the first-order $\la$-formula $x\cdot y=y\cdot x$.  Then $\mathfrak{G}\models m_x<r.\varphi(g)$ if and only if the cardinality of $C_G(g)$, the centralizer of $g$ in $G$, is $<r\cdot m$.  In particular, if $r<1$, then $\mathfrak{G}\not\models \forall y m_x<r.\varphi$ as the centralizer of the identity has cardinality $m$.
\end{ex}




As in classical logic, we say that $D\subseteq M^n$ is \emph{definable (in $\M$)} if there is a formula $\varphi(\vec x, \vec y)$, with $|\vec x|=n$, and $\vec a\in M^{|\vec y|}$ such that $D=\varphi(M,\vec a)$.  If $A\subseteq M$ and the tuple $\vec a$ as above lies in $A$, we also say that $D$ is \emph{$A$-definable (in $\M$)} or \emph{definable (in $\M$) over $A$}.  We let $\dfbl_n(M)$ denote the Boolean algebra of definable subsets of $M^n$.  A function $h:M^m\rightarrow M^n$ is \emph{definable (in $\M$)} if the graph of $h$ is a definable subsets of $M^{m+n}$.

\begin{conv}
From now on, when we say that $\M$ is an $\la$-(quasi)structure, we mean that $\M$ is an $\la$-(quasi)structure of $\AML$.  We use fraktur letters $\M$ and $\N$ (sometimes decorated with subscripts) to denote quasistructures.  The corresponding roman letter denotes the underlying universe of the quasistructure.  (So, for example, the underlying universe of $\M$ will be $M$ and the underlying universe of $\N_i$ will be $N_i$.)  Likewise, $\la$-formulae will always mean $\AML$ $\la$-formulae.  If we wish to consider a structure from classical first-order logic, then we will speak of a \emph{classical} structure and denote it using calligraphic letters such as $\mathcal{M}$.  Likewise, we will speak of \emph{classical} formulae.
\end{conv}

Suppose $\la'$ is signature extending the signature $\la$.  Suppose that $\M$ is an $\la'$-quasistructure.  Then the $\la$-\emph{reduct} of $\M$, denoted $\M|\la$, is the $\la$-quasistructure obtained from $\M$ by ``forgetting'' to interpret symbols from $\la'\setminus \la$.  We also refer to $\M$ as an \emph{expansion} of $\M|\la$.

\begin{df}
Suppose that $\M$ and $\N$ are $\la$-structures.  We say that $\M$ is a \emph{substructure} of $\N$ if:
\begin{enumerate}
\item the first-order part of $\M$ is a substructure of the first-order part of $\N$, and
\item for each $B\in \B^{\N}_n$, we have $B\cap M^n\in \B^{\M}_n$.
\end{enumerate}
We further say that $\M$ is an \emph{elementary substructure} of $\N$ if, for all formulae $\varphi(\vec x)$ and $a\in M^{|\vec x|}$, we have $\M\models \varphi(\vec a)$ if and only if $\N\models \varphi(\vec a)$.
\end{df}

It may seem a bit strange that the notion of substructure does not mention the measure part of the structures in question.  However, given that the constructor $m_{\vec x}<r$ acts as a quantifier and arbitrary substructures need not respect the truth of quantified formulae, this lack of a requirement need not be so surprising.

We define the notions of \emph{embedding} and \emph{elementary embedding} in the obvious way so as to align with our notion of substructure and elementary substructure.

\

\section{The Canonical Measures on an \texorpdfstring{$\aleph_1$}{}-saturated \texorpdfstring{$\AML$}{} structure}
\label{sec:canonical}

An extremely important source of $\AML$ structures are those that are equipped with an actual measure.

\begin{df}
A \emph{measured $\la$-structure} is an $\AML$ $\la$-structure $\mathfrak{M}$ such that there is a measure $\nu^n$ on a $\sigma$-algebra $\B^n$ of subsets of $M^n$ extending $\dfbl_n(M)$ such that:
\begin{enumerate}
\item For definable $B\subseteq M^n$, we have $\nu^n(B)=\mu^\M_n(B)$.
\item For each $m,n\geq 1$, the measure $\nu^{m+n}$ is an extension of the product measure $\nu^m\otimes \nu^n$ 
\item (Fubini property)  For every $m,n\geq 1$ and $B\in \B^{m+n}$, we have
\begin{enumerate}
\item for almost all $x\in M^m$, $B^x\in \B^n$;
\item for almost all $y\in M^n$, $B_y\in \B^m$;
\item the functions $x\mapsto \nu^n(B^x)$ are $y\mapsto \nu^m(B_y)$ are $\nu^m$- and $\nu^n$-measurable, respectively, and $$\nu^{m+n}(B)=\int \nu^n(B^x)d\nu^m(x)=\int \nu^m(B_y)d\nu^n(y).$$
\end{enumerate} 
\end{enumerate}
\end{df}

For example, any finite $\la$-structure equipped with its normalized counting measure is a measured structure; in fact, these structures (and their ultraproducts) are our primary examples of measured structures.

Even if an $\AML$ structure does not come equipped with a measure, often there is still a ``natural'' measure that can be placed on the structure, a process that we now describe.  First, we say that an $\la$-structure $\M$ is \emph{$\aleph_1$-compact} if whenever $n\geq 1$ and $(D_m \ : \ m\in \n)$ is a family of definable subsets of $M^n$ with the finite intersection property, then $\bigcap_{m\in \n}D_m\not=\emptyset$.  (If the signature $\la$ is countable, then this coincides with the notion of an \emph{$\aleph_1$-saturated structure} more commonly encountered in model theory; when the signature is countable, we will often used the term saturated rather than compact.)  For example, any finite $\la$-structure is $\aleph_1$-compact; many ultraproducts are also $\aleph_1$-compact.  In this section, we will show that certain $\aleph_1$-compact $\AML$ structures can be equipped with a family of canonical measures, the so-called \emph{Loeb measures}.

We write $\sigma(\mathcal{B})$ for the $\sigma$-algebra generated by the algebra $\mathcal{B}$.

Suppose that $\M$ is an $\aleph_1$-compact $\AML$ $\la$-structure and $(A_m \ : \ m\in \n)$ is a family of definable subsets of $M^n$ for which $\bigcup_{m\in \n} A_m$ is also definable.  Then, by $\aleph_1$-compactness, we have that $\bigcup_{m\in \n} A_m=\bigcup_{m=0}^k A_m$ for some $k\in \n$.  Consequently, $\mu^\M_n|\dfbl_n(M)$ is a pre-measure.  Thus, by the Caratheodory extension theorem, there is a measure, which we again call $\mu^\M_n$ or simply $\mu_n$, on $\sigma(\dfbl_n(M))$, extending the original $\mu^\M_n$ given by $$\mu^n(B):=\inf\{\sum_m \mu^\M_n(D_m) \ : \ B\subseteq \bigcup_mD_m, D_m\in \dfbl_n(M) \text{ for all }m\}.$$  Moreover,  if $\mu^\M_n|\sigma(\dfbl_n(M))$ is $\sigma$-finite, then $\mu^n$ is the unique measure on $\sigma(\dfbl_n(M))$ extending the original $\mu^\M_n$. Observe, however, that by $\aleph_1$-compactness again, $\mu^\M_n|\dfbl_n(M)$ is $\sigma$-finite if and only if $\mu^\M_n|\dfbl_n(M)$ is finite (which occurs if and only if $\mu^\M_1(M)$ is finite by the product property).  We refer to the family of measures $(\mu^\M_n)$ as the \emph{canonical measures} on $\M$.  Although they are defined for all $\aleph_1$-compact $\la$-structures, we mainly consider them in the case when $\mu^\M_1(M)$ is finite, for then they are truly ``canonical.''  (Indeed, when a pre-measure on an algebra $\mathcal{A}$ is not $\sigma$-finite, there are many extensions of it to a measure on $\sigma(\mathcal{A})$; see Exercise 1.7.8 in \cite{taomeasure}.)

Until further notice, we assume that $\mu^\M_1|\dfbl_1(M)$ is a finite pre-measure on $\dfbl_1(M)$ (and hence $\mu^\M_n|\dfbl_n(M)$ is a finite pre-measure on $\dfbl_n(M)$ for all $n\geq 1$).  Let $\sigma(\dfbl_m(M))\otimes \sigma(\dfbl_n(M))$ denote the product $\sigma$-algebra and $\mu_m\otimes \mu_n$ denote the product measure.  Observe that $\sigma(\dfbl_m(M))\otimes \sigma(\dfbl_n(M))$ is generated by sets of the form $A\times B$, with $A\in \dfbl_m(M))$ and $B\in \dfbl_n(M))$; since $A\times B\in \dfbl_{m+n}(M)$ for such $A$, $B$, we have $\sigma(\dfbl_m(M))\otimes \sigma(\dfbl_n(M))\subseteq \sigma(\dfbl_{m+n}(M))$.  Now observe that, if $A\in \dfbl_m(M)$ and $B\in \dfbl_n(M)$, then $\mu_{m+n}(A\times B)=\mu_m(A)\cdot \mu_n(B)$ by the Product property.  
Thus $$E:=\{D\in \sigma(\dfbl_m(M))\otimes \sigma(\dfbl_n(M)) \ : \ \mu_{m+n}(D)=(\mu_m\otimes \mu_n)(D)\}$$ contains all sets of the form $A\times B$, where $A\in \dfbl_m(M)$ and $B\in \dfbl_n(M)$. One can easily check that $E$ is a $\lambda$-class.  Since the sets of the form $A\times B$, where $A\in \dfbl_m(M)$ and $B\in \dfbl_n(M)$, form a $\pi$-class, we have, by the $\pi$-$\lambda$ Theorem, that $\sigma(\dfbl_m(M))\otimes \sigma(\dfbl_n(M))\subseteq E$.  (For the definitions of $\lambda$- and $\pi$-classes as well as the statement of the $\pi$-$\lambda$-theorem, see Section 17.A of \cite{kechris}.)  Consequently, we see that $\mu_m\otimes \mu_n=\mu_{m+n}|\sigma(\dfbl_m(M))\otimes \sigma(\dfbl_n(M))$.  This proves:

\begin{lemma}
If $\M$ is an $\aleph_1$-compact $\la$-structure such that $\mu^\M(M)$ is finite, then $\M$, equipped with its canonical measures, satisfies the first two axioms in the definition of a measured structure. 
\end{lemma} 

The third axiom, requiring that Fubini's Theorem hold, might still fail for sets which do not belong to $\sigma(\dfbl_m(M))\otimes \sigma(\dfbl_n(M))$.  We call an $\aleph_1$-compact $\la$-structure whose canonical measures satisfy the Fubini property a \emph{Fubini structure}.  (The measurability of the functions in (3) always holds for the canonical measures, so really a Fubini structure is one where the displayed equation in the statement of the Fubini property is required to hold.)  Thus, by the preceding lemma, a Fubini structure, equipped with its canonical measures, is a measured $\la$-structure.

The following situation is also important:

\begin{df}
Suppose that $\mathcal{M}$ is a classical $\la$-structure which is equipped with a family of measures $\nu^n$ on a $\sigma$-algebra $\B^n$ of subsets of $\M^n$ extending the algebra of definable sets and such that $\nu^{m+n}$ is an extension of $\nu^m\otimes \nu^n$ for all $m,n\geq 1$.  We call $(\mathcal{M},(\nu^n))$ a \emph{classical measured structure}.  Then we can turn $\mathcal{M}$ into an $\la$-quasistructure $\M$ by declaring, for $B\subseteq M^n$, that $v^\M_n(B)=\odot$.  We call $\M$ the \emph{$\AML$ quasistructure associated to $(\mathcal{M},(\nu^n))$}.  
\end{df}

\begin{rmk}
Usually, the AML quasistructure associated to a classical measured structure is not a structure: due to the added expressivity power, there are new definable sets which may not be measurable.  (However this quasistructure will satisfy the rank $0$ measurability condition, and so satisfaction for formulas of rank $1$ can be defined.)  For finite classical measured structures (which are the structures most important for our combinatorial applications),  the associated AML quasistructure is in fact a structure because all sets are measurable.
\end{rmk}

\begin{rmk}
Consider $\M$, an $\aleph_1$-compact $\la$-structure with $\mu^\M_1(M)$ finite, equipped with its family $(\mu^\M_n)$ of canonical measures.  Let $\mathcal{M}$ be the classical $\la$-structure obtained from $\M$ by considering only the interpretations of the symbols in $\la$.  Then $(\mathcal{M},(\mu^\M_n))$ is a classical measured structure.  Let $\M'$ be the $\AML$ quasistructure associated to $(\mathcal{M},(\mu^\M_n))$.  How do $\M$ and $\M'$ compare?  Suppose that $\varphi(x)$ is a classical $\la$-formula and $r\in \q^{\geq 0}$.  It is straightforward to see that $\M\models m_x\leq r.\varphi(x)$ implies $\M'\models m_x\leq r.\varphi(x)$.  However, the converse need not hold.  For example, suppose that $\M'\models m_x\leq 0.\varphi(x)$, whence $\mu(\varphi(M))=0$.  However, if $v^\M(\varphi)=\oplus$, then $\M\models \neg m_x\leq 0.\varphi(x)$.  (Such phenomena can, and will, happen in ultraproducts, as we will see in later sections.)  Thus, in some sense, $\M'$ is a ``completion'' of $\M$.
\end{rmk}

\noindent We now briefly turn to the question of what our logic can say about these canonical measures.  We present a couple of examples.  Suppose that $(X,\mathcal{S})$ is a measurable space such that $\{x\}\in \mathcal{S}$ for every $x\in X$.  (This happens, for example, when $(X,\mathcal{S})=(M^n,\sigma(\dfbl_n(M)))$ for $\M$ an $\la$-structure.)  Then a measure $\mu$ on $X$ is said to be \emph{continuous} if $\mu(\{x\})=0$ for each $x\in X$.  The following is immediate.

\begin{prop}\label{cont}
Suppose that $\M$ is an $\aleph_1$-compact $\la$-structure.  Then we have $\M\models \{\forall x m_y\leq q(x=y) \ : \ q\in \q^{>0}\}$ if and only if $\mu^1$ is a continuous measure.
\end{prop} 


More generally, there is a set $\Gamma(x)$ consisting of $\la$-formulae containing a single free variable $x$ such that, for every $\la$-structure $\M$ and every $a\in M$, $$\mu^1(\{a\})=0\Leftrightarrow \M\models \varphi(a) \text{ for all }\varphi\in \Gamma.$$  Indeed, take $\Gamma(x):=\{m_y<q.(x=y) \ : \ q\in \q^{>0}\}$.  In model-theoretic terms, the set of elements of $\mu^1$-measure $0$ is uniformly type-definable in all $\aleph_1$-compact $\la$-structures.  Below, as a consequence of the ultraproduct construction, we will see that this set is not uniformly definable in all $\la$-structures.

Here is another proposition along the same lines.

\begin{prop}
There is a set $T$ of $\AML$ sentences such that, for any $\aleph_1$-compact $\la$-structure $\M$ with $\mu^\M_1(M)$ finite, $\M\models T$ if and only if $\mu^1$ is a probability measure.
\end{prop}

\begin{proof}
Take $T:=\{m_x\leq 1(x=x)\}\cup \{\neg m_x\leq q(x=x) \ : \ q\in (0,1)\cap \q\}$.
\end{proof}

\begin{prop} 
There is a set $T$ of $\AML$ sentences such that, for any $\aleph_1$-compact $\la$-structure $\M$ with $\mu^\M_1(M)$ finite, $\M\models T$ if and only if $\M$ is a Fubini structure.
\label{ensuring_fubini}
\end{prop}
\begin{proof} 
We take $T$ to consist of the following two schemes:
\begin{enumerate}
\item For every $\varphi(\vec x,\vec y,\vec x),\psi(\vec x,\vec z)$ and $q,r,t$ with $t<qr$,
\[\forall \vec z\left[\left(\forall \vec x (\psi\rightarrow m_{\vec y}\geq r. \varphi)\wedge m_{\vec x}\geq q.\psi\right)\rightarrow m_{\vec x,\vec y} >t.(\varphi\wedge\psi)\right],\]
\item For every $\varphi(\vec x,\vec y,\vec x),\psi(\vec x,\vec z)$ and $q,r,t$ with $qr<t$,
\[\forall\vec z\left[\left( \forall \vec x (\psi\rightarrow m_{\vec y}\leq r. \varphi)\wedge m_{\vec x}\leq q.\psi\right)\rightarrow m_{\vec x,\vec y}< t.(\varphi\wedge\psi)\right].\]
\end{enumerate}
  Note that $\psi$ may not contain variables from $\vec y$. 

Suppose $\M\vDash T$.  It suffices to show that for any definable set $B\subseteq M^{m+n}$, we have $\mu^{m+n}(B)=\int\mu^n(B^{\vec x})d\mu^m(\vec x)$.  Fix a formula $\varphi(\vec x,\vec y,\vec z)$ and $\vec c\in M^{|\vec z}$ so that $B=\varphi(M,\vec c)$.  Given rationals $r,r'$, we set $B_{r,r'}=\{\vec x\setmid \M\vDash m_{\vec y}\geq r. \varphi(\vec x,\vec y,\vec c)\wedge m_{\vec y}<r'. \varphi(\vec x,\vec y,\vec c)\}$; we will also write $B_{r,r'}$ as an abbreviation for its defining formula.  

Observe that we can approximate the function $\vec x\mapsto\mu^n(B^{\vec x})$ by step functions of the form $\sum_i r_i\chi_{B_{r_i,r_{i+1}}}$ for sequences of rationals $r_1<\cdots<r_k$.  Let $r_1<\cdots<r_k$ be such a sequence of rationals.  For any rational $q_i<\mu^m(B_{r_i,r_{i+1}})$, we must have $\M\vDash m_{\vec x}\geq q_i. B_{r_i,r_{i+1}}$.  Since $\M\vDash T$, we must have 
\[\M\vDash  m_{\vec x,\vec y}> t. (\varphi\wedge B_{r_i,r_{i+1}})\]
for every $t<q_ir_i$, and so $\mu^{m+n}(\{(\vec x,\vec y)\in B\setmid \vec x\in B_{r_i,r_{i+1}}\})\geq q_ir_i$.  Since the $B_{r_i,r_{i+1}}$ are pairwise disjoint, it follows that $\mu^{m+n}(B)\geq\sum_i q_ir_i$.  Since we may choose $q_i,r_i$ as above so that $\sum_iq_ir_i$ is arbitrarily close to $\int \mu^n(B^{\vec x})d\mu^m(\vec x)$, we have $\mu^{m+n}(B)\geq\int \mu^n(B^{\vec x})d\mu^m(\vec x)$.
 
Similarly, for any $q_i>\mu^m(B_{r_i,r_{i+1}})$, we must have $\M\vDash m_{\vec x}\leq q_i. B_{r_i,r_{i+1}}$, and since $\M\vDash T$, we also have 
\[\M\vDash m_{\vec x,\vec y}< t. (\varphi\wedge B_{r_i,r_{i+1}})\]
for any $t>q_ir_{i+1}$.  Therefore $\mu^{m+n}(\{(\vec x,\vec y)\in B\setmid \vec x\in B_{r_i,r_{i+1}}\})\leq q_ir_{i+1}$, and so $\mu^{m+n}(B)\leq\sum_i q_ir_{i+1}$.  Again, since $\sum_iq_ir_{i+1}$ can be made arbitrarily close to $\int \mu^n(B^{\vec x})d\mu^m(\vec x)$, we have $\mu^{m+n}(B)\leq\int \mu^n(B^{\vec x})d\mu^m(\vec x)$.

Now suppose $\M$ is Fubini; we must show that every sentence in $T$ holds.  For the first family of sentences, suppose $\M\vDash\forall \vec x (\psi\rightarrow m_{\vec y}\geq r. \varphi)\wedge m_{\vec x}\geq q.\psi$ (with parameters $\vec c\in M^{|\vec z|}$).  Setting $C=\psi(M,\vec c)$, we see that $\mu(C)\geq q$.  Let $B=\varphi(M,\vec c)$.  For each $\vec a\in C$, $\mu(B^{\vec a})\geq r$.  Since $\mu$ is Fubini, $\mu^{m+n}(B)=\int_C\mu^n(B^{\vec x})d\mu^n(\vec x)\geq qr$, and therefore for each $t<qr$ we have $\M\vDash m_{\vec x}>t. (\varphi\wedge\psi)$.

For the second family of sentences, suppose $\M\vDash\forall \vec x (\psi\rightarrow m_{\vec y}\leq r. \varphi)\wedge m_{\vec x}\leq q.\psi$.  Again setting $C=\psi(M,\vec c)$ and $B=\varphi(M,\vec c)$, for each $\vec a\in C$, $\mu(B^{\vec a})\leq r$.  Since $\mu$ is Fubini, $\mu^{m+n}(B)=\int_C\mu^n(B^{\vec x})d\mu^n(\vec x)\leq qr$, and therefore for each $t>qr$ we have $\M\vDash m_{\vec x}<t. (\varphi\wedge\psi)[s]$.
\end{proof}

\begin{rmk}\label{fubinisound}
Fix $T$ as in the proof of the previous proposition.  The proof we have given above shows that $\M\models T$ for any measured structure $\M$. 
\end{rmk}

\section{Ultraproducts}
\label{sec:ultraproducts}

We suppose that the reader is familiar with the standard ultraproduct construction from first-order logic.  We sometimes use the phrase ``$P(i)$ holds a.e.'' to mean that the set of $i$'s for which $P(i)$ holds belongs to the ultrafilter.

Suppose that $(\M_i \ : \ i\in I)$ is a family of $\la$-structures and suppose that $\U$ is an ultrafilter on $I$.  We let $\M$ denote the classical ultraproduct of the family $(\M_i)$ with respect to $\cU$.  We declare $A\in \B^\M_n$ if and only if $A=\prod_\U A_i$, where each $A_i\in \B^{\M_i}_n$; in this case, we declare $\mu^\M_n(A)=\lim_\U \mu^{\M_i}_n(A_i)$ (the \emph{ultralimit} of $\mu^{\M_i}_n(A_i)$).  Observe that $\lim_\U(\mu^{\M_i}_n(A_i))$ always exists and is unique since $[0,\infty]$ is a compact hausdorff space.  

If $\mu^\M_n(A)=r\in \q$, we set
\[
v^\M_n(A)=
\begin{cases}
\oplus \quad \text{ if $\mu^{\M_i}_n(A_i)>r$ a.e.}\\
\ominus \quad \text{ if $\mu^{\M_i}_n(A_i)<r$ a.e.}\\
\oplus \quad \text{ if $\mu^{\M_i}_n(A_i)=r$ and $v^{\M_i}_n(A_i)=\oplus$ a.e.}\\
\ominus \quad \text{ if $\mu^{\M_i}_n(A_i)=r$ and $v^{\M_i}_n(A_i)=\ominus$ a.e.}\\
\odot \quad \text{ if $\mu^{\M_i}_n(A_i)=r$ and $v^{\M_i}_n(A_i)=\odot$ a.e.}
\end{cases}
\]

It is reasonably straightforward to check that each $\B^\M_n$ is an algebra and $\mu^\M_n$ is a finitely additive measure on $\B^\M_n$, whence the ultraproduct $\M$ is a quasistructure.

Note that the assignment of $\nu^{\M}_n(A)$ is consistent with the interpretations of $\oplus$ and $\ominus$ as indicating approximation from above or below: if $\mu^{\M_i}_n(A_i)>\mu^{\M}_n(A)$ a.e., we have obtained the measure $\mu^{\M}_n(A)$ by approximations from above, thus (irrespective of whether the values of $\mu^{\M_i}_n(A_i)$ were obtained exactly or obtained by approximation themselves) $v^{\M}_n(A)=\oplus$; if $\mu^{\M_i}_n(A_i)=r$ and $v^{\M_i}_n(A_i)=\oplus$ a.e. then almost every $\mu^{\M_i}_n(A_i)$ was found by approximations from above, so the measure of $\mu^{\M}_n(A)$ is found the same way.

The following theorem is of fundamental importance for our applications of $\AML$.

\begin{thm}[\L o\'s Theorem for $\AML$]
Suppose that $(\M_i \setmid i\in I)$ is a family of $\la$-structures, $\U$ is a ultrafilter on $I$, and $\M=\prod_{\cU}\M_i$ is the ultraproduct of the family $(\M_i)$.  Let $\varphi$ be an $\la$-formula and $s$ a valuation on $\M$.  For each $i\in I$, let $s_i$ be a valuation on $\M_i$ such that, for each variable $v$, we have $s(v)=[s_i(v)]_\U$.
Then $\M\models \varphi[s]\Leftrightarrow \M_i\models \varphi[s_i] \text{ for }\U\text{-almost all }i.$
\end{thm}

\begin{proof}
As in the proof of the classical \L os theorem, we proceed by induction on the complexity of formulae.  We only need to explain how to deal with the measure constructor cases, that is, we consider a formula $\varphi(\vec x,\vec y)$ such that the theorem holds true for $\varphi$ and we show that it holds for $m_{\vec x}\bowtie r.\varphi(\vec x,\vec y)$, where $\bowtie \in \{<,\leq\}$.  We fix valuations $s_i$ and $s$ as in the statement of the theorem and set $\vec b_i:=s_i(\vec y)$ and $\vec b:=s(\vec y)$.  We will repeatedly use the inductive hypothesis that $\varphi(M,\vec b)=\prod_\U \varphi(M_i,\vec b_i)$.

If $\mu^{\M_i}(\varphi(M_i,\vec b_i))<r$ a.e. then $\M_i\vDash m_{\vec x}<r. \varphi(\vec x,\vec b_i)$ and $\M_i\vDash m_{\vec x}\leq r. \varphi(\vec x,\vec b_i)$ a.e.; also either $\mu^{\M}(\varphi(M,\vec b))<r$ or both $\mu^{\M}(\varphi(M,\vec b))\leq r$ and $v^{\M}(\varphi,\vec b)=\ominus$, and in either case $\M\vDash m_{\vec x}<r.\varphi(\vec x,\vec b)$ and $\M\vDash m_{\vec x}\leq r.\varphi(\vec x,\vec b)$.

If $\mu^{\M_i}(\varphi(M_i,\vec b_i))>r$ a.e. then, symmetrically, we have $\M_i\vDash m_{\vec x}>r. \varphi(\vec x,\vec b_i)$ and $\M_i\vDash m_{\vec x}\geq r. \varphi(\vec x,\vec b_i)$ a.e., and also $\M\vDash m_{\vec x}>r.\varphi(\vec x,\vec b)$ and $\M\vDash m_{\vec x}\geq r.\varphi(\vec x,\vec b)$.

Finally, if $\mu^{\M}(\varphi(M_i,\vec b_i))=r$ a.e. then $v^{\M}(\varphi(M,\vec b))=v^{\M_i}(\varphi(M_i,\vec b_i))$ a.e., so $\M\vDash m_{\vec x}<r. \varphi(\vec x,\vec b)$ iff $\M_i\vDash m_{\vec x}<r. \varphi(\vec x,\vec b_i)$ a.e. and $\M\vDash m_{\vec x}\leq r. \varphi(\vec x,\vec b)$ iff $\M_i\vDash m_{\vec x}\leq r. \varphi(\vec x,\vec b_i)$.
\end{proof}

\begin{cor}
The ultraproduct of structures is a structure.
\end{cor}

\begin{proof}
Definable subsets of the ultraproduct are ultraproducts of definable sets in the factor structures.
\end{proof}

\begin{cor}
With the same hypotheses as in the previous theorem, if $\sigma$ is an $\la$-sentence, then $$\M\models \sigma \Leftrightarrow \M_i\models \sigma \text{ for }\U\text{-almost all }i.$$
\end{cor}

\begin{cor}[Compactness Theorem for $\AML$]
If $\Sigma$ is a set of $\la$-sentences such that each finite subset of $\Sigma$ has a model, then $\Sigma$ has a model.
\end{cor}

\begin{proof}
The usual ultraproduct proof of the Compactness theorem from the \L o\'s theorem applies; see, for example, Corollary 4.1.11 of \cite{changkeisler}.
\end{proof}

\begin{cor}
For any $\la$-structure $\M$, any set $I$, and any nonprincipal ultrafilter $\U$ on $I$, the diagonal embedding $j:\M\to \M^\U$ given by $j(a)=[(a)]_\U$ is an elementary embedding.
\end{cor}

The following result is standard in the classical context; this proof carries over immediately to the framework of $\AML$.

\begin{prop}
Suppose that $\cU$ is a countably incomplete ultrafilter on a set $I$ and suppose that $\la$ is countable.  Suppose that $(\M_i \ : \ i\in I)$ is a family of $\la$-structures.  Then $\prod_\cU \M_i$ is $\aleph_1$-saturated.  In particular, if $I=\n$ and $\cU$ is any nonprincipal ultrafilter on $\n$, then $\prod_{\cU}\M_i$ is $\aleph_1$-saturated.
\end{prop}

Consequently, an ultraproduct $\M$ for which $\mu^\M_1$ is finite can be equipped with its canonical family of measures.  (Observe that, by the \L o\'s theorem, $\mu^\M_1$ is finite if and only if the sequence $(\mu^{\M_i}_1)$ is ``essentially bounded.'')  These measures are instrumental for our applications of $\AML$ to problems in finite combinatorics and ergodic theory.

We can also use the \L o\`s theorem to show that properties of the canonical measures on the factor models are inherited by the canonical measures on the ultraproduct.

\begin{cor} 
Suppose that $(\M_i\setmid i\in I)$ is a family of Fubini structures, $\cU$ is a nonprincipal ultrafilter on $I$, and $\M=\prod_{\cU}\M_i$.  Suppose further that there is a $c\in \r^{>0}$ such that $\mu^\M_1(M_i)\leq c$ for $\U$-almost all $i$.  Then $\M$ is a Fubini structure.
\end{cor}

\begin{proof}
This is immediate from Proposition \ref{ensuring_fubini} and the \L o\'s theorem.
\end{proof}

\begin{cor}
Suppose that $(\M_i \setmid i\in I)$ is a family of $\aleph_1$-compact $\la$-structures such that $\mu^{\M_i}_1$ is a continuous measure for almost all $i$. Suppose that $\cU$ is a nonprincipal ultrafilter on $I$ and $\M=\prod_{\cU}\M_i$.  
Then $\mu^\M_1$ is also a continuous measure.  
\end{cor}

\begin{proof}
This is immediate from Proposition \ref{cont} and the \L o\'s theorem.
\end{proof}

\begin{rmk}
Under some set-theoretic assumptions, the converse of the above corollary holds.  A nonprincipal ultrafilter $\cU$ on a (necessarily uncountable) index set $I$ is said to be \emph{countably complete} if whenever $(D_m\setmid m\in \n)$ is a sequence of sets from $\cU$ such that $D_m\supseteq D_{m+1}$ for all $m$, we have $\bigcap_{m\in \n}D_m\in \cU$.  With the same hypotheses as in the previous corollary, if we further assume that $\cU$ is countably complete, then continuity of $\mu^\M_1$ implies continuity of almost all $\mu^{\M_i}_1$.  The existence of a countably complete ultrafilter is tantamount to the existence of a \emph{measurable cardinal}, which is a fairly mild set-theoretic assumption that most set theorists are willing to assume when necessary.  However, it is straightforward to check that if $\cU$ is countably complete and each $\M_i$ is finite, then $\prod_{\cU}\M_i$ will also be finite, making such ultraproducts unusable for many purposes.
\end{rmk}

\begin{cor}
There does not exist an $\la$-formula $\varphi(x)$ such that for every $\aleph_1$-compact $\la$-structure $\M$, we have
$$\mu^\M_1(\{a\})=0 \Leftrightarrow \M\models \varphi(a).$$
\end{cor}

\begin{proof}
Suppose, towards a contradiction, that such an $\la$-formula $\varphi(x)$ exists.  For $i\geq 1$, let $\mathfrak{M}_i$ be the classical measured $\la$-structure whose universe is $\{1,\ldots,i\}$, whose measures are given by the normalized counting measures, and which interprets the $\la$-structure in an arbitrary fashion.  Let $\M_i$ be the $\AML$ $\la$-structure associated with $\mathfrak{M}_i$.  Let $\U$ be a nonprincipal ultrafilter on $\n^{>0}$ and let $\M=\prod_\U \M_i$.  
It is easy to see that $\mu^\M_1(M)\leq 1$ and $\mu^\M_1(\{a\})=0$ for any $a\in M$.  It follows that $\M\models \forall x\varphi(x)$.  By \L o\'s, we have that $\M_i\models \forall x\varphi(x)$ for some $i\in\n^{>0}$.  Since $\M_i$ is $\aleph_1$-compact, it follows that each $m\in M_i$ has measure $0$, which is a contradiction.
\end{proof}



\begin{cor}
There does not exist an $\la$-sentence $\varphi$ such that, whenever $\M$ is an $\aleph_1$-compact structure with $\mu^\M_1(M)$ finite, then $\M\models \varphi$ if and only if $\mu^\M_1$ is a probability measure.
\end{cor}

\begin{proof}
Suppose, towards a contradiction, that such a sentence $\varphi$ exists.  For $n\geq 1$, let $\M_n$ be a finite $\la$-structure such that $\mu^\M_1(M_n)=1-\frac{1}{n}$.  Let $\M:=\prod_{\cU}\M_n$, where $\cU$ is a nonprincipal ultrafilter on $\n^{>0}$.  Observe that $\mu^\M_1(M)=1$, whence $\M\models \varphi$.  Thus, $\M_n\models \varphi$ for some $n$, contradicting the fact that $\mu^{\M_n}1(M_n)<1$.
\end{proof}

\begin{cor}
There does not exist a set $T$ of $\la$-sentences such that, for any $\la$-structure $\M$, $\M\models T$ if and only if $\mu^\M_1$ is $\sigma$-finite.
\end{cor}

\begin{proof}
Suppose, towards a contradiction, that there is a set $T$ of $\la$-sentences such that $\M\models T$ if and only if $\mu^\M_1$ is $\sigma$-finite.  Let $M$ be a countably infinite set equipped with the unnormalized counting measure $\nu(B)=|B|$ and turn $M$ into an $\AML$ $\la$-structure $\M$ by interpreting the symbols in $\la$ in an arbitrary fashion.  Note that $\M\models T$ because $\M$ is the union of its singletons, which are definable and have measure 1.  Let $\cU$ be a nonprincipal ultrafilter on $\n$ and let $\N:=\M^{\cU}$.  Then, by the \L o\'s theorem, we have that $\N\models T$, so $\mu^\N_1$ is $\sigma$-finite.   Since $\N$ is $\aleph_1$-compact, we have that $\mu^\N_1$ is a finite measure.  Thus, for some $n\in \n$, $\N\models \neg \exists x_1\cdots \exists x_n (\bigwedge_{i=1}^n\neg m_x<1.(x=x_i))$.  Consequently, $\M\models \neg \exists x_1\cdots \exists x_n (\bigwedge_{i=1}^n\neg m_x<1.(x=x_i))$, contradicting the fact that $M$ contains infinitely many singletons of measure $1$.  
\end{proof}


We now turn to the connection between integration and ultraproducts.  
Suppose that $(\M_i \setmid i\in I)$ is a family of $\la$-structures and $h_i:M_i^m\to M_i^n$ are functions.  We will say that the family $(h_i\setmid i\in I)$ is \emph{uniformly definable} if there exists a formula $\varphi(\vec x,\vec y, \vec z)$, where $|\vec x|=m$ and $|\vec x|=n$, and tuples $\vec c_i\in M_i^{|\vec z|}$ such that, for all $i\in I$, $\vec a\in M_i^m$ and $\vec b\in M_i^n$, we have $h_i(\vec a)=\vec b$ if and only if $\M_i\models \varphi(\vec a,\vec b, \vec c_i)$.  If $\M=\prod_{\cU} \M_i$ and $\vec c:=[\vec c_i]_{\cU}$, then by \L o\'s, we have that $\varphi(\vec x,\vec y,\vec c)$ defines a function $h:M^m\to M^n$, which we call the \emph{ultraproduct} of the uniformly definable family $(h_i)$.

We will also need to consider families of functions.  Suppose $(\M_i \setmid \ i\in I)$ is a family of measured $\la$-structures, and, for each $i$, $g_i:M^m_i\rightarrow \r$ is a function.  We then say the family $(g_i\setmid\ i\in I)$ is \emph{uniformly layerwise definable}\footnote{This notion of definability of a function into $\mathbb{R}$ differs from the usual notion in the literature, which requires the inverse image of a closed set be type-definable.} if, for each $q\in \mathbb{Q}$, there is a formula $\varphi_q(\vec x,\vec z)$ and tuples $\vec c_i\in M_i^{|\vec z|}$ such that, for $\vec a\in M_i^m$, $g_i(\vec a)<q$ if and only if $\M_i\models \varphi_q(\vec a,\vec c_i)$. 

In the following theorem, the bound of $1$ can easily be replaced by any positive real number.
\begin{thm}\label{intultra2}
Suppose that $(\M_i \setmid \ i\in I)$ is a family of measured $\la$-structures with $\mu^1_i(M_i)\leq 1$ for each $i\in I$.  Suppose that $\U$ is a nonprincipal ultrafilter on $I$ and $\M:=\prod_\U \M_i$ is equipped with its canonical measures $(\mu^m)$.  For each $i\in I$, let $g_i:M^n_i\rightarrow[-1,1]$ be a function and suppose that the family $(g_i \ : \ i\in I)$ is uniformly layerwise definable.  Then there is an expansion $\M'$ of $\M$ to a structure and a measurable (with respect to $\mathcal{B}^{\M',n}$) function $g:M^n\rightarrow[-1,1]$ such that, for every uniformly definable family $(h_i\setmid i\in I)$ of functions of $m$ variables with ultraproduct $h$, we have $\int (g\circ h)\ d\mu^m=\lim_\cU \int (g_i\circ h_i)\ d\mu_i^m$.
\label{ultraproduct_of_integrals}
\end{thm}

\begin{proof} 
Consider the extension $\la'$ of $\la$ obtained by adding, for each $q\in\mathbb{Q}\cap[-1,1]$, a new unary predicate $G_q$.  We define expansions $\M'_i$ of $\M_i$ by interpreting $G_q^{\M'_i}=\{\vec x\in M_i^n\setmid g_i(\vec x)<q\}$; since this is a definitional expansion, $\M_i'$ is indeed an $\la'$-structure.  Let $\M':=\prod_\U \M'_i$, and define $g(\vec x)=\inf\{q\in \q^{>0}\setmid \M'\vDash G_q(\vec x)\}$.  It is clear that $g$ is $\B^{\M',n}$-measurable.

Now suppose that $(h_i\setmid i\in I)$ is a uniformly definable family of functions, say defined by $\varphi(\vec x,\vec y)$ (and some parameters, which we suppress for sake of exposition), with ultraproduct $h$.  Set $m:=|\vec x|$ and $n:=|\vec y|$.  
Fix $c\in \q^{>0}$ and choose a partition $-1=q_0<q_1<\cdots<q_k=1$ of $[-1,1]$ by rational numbers such that that $|q_j-q_{j+1}|<1/c$ for each $j<k$.  Fix also some rational $q_{k+1}>1$.  For $0\leq j\leq k$, set 
\[S_i^{q_j,q_{j+1}}:=\{\vec x\in M_i^m\setmid q_j\leq (g_i\circ h_i)(\vec x)< q_{j+1}\}\]
and take $S^{q_j,q_{j+1}}=\prod_\U S_i^{q_j,q_{j+1}}$, whence $\mu^{m}(S^{q_j,q_{j+1}})=\lim_\U \mu^{m}_i(S_i^{q_j,q_{j+1}})$.

We have by definition that
\[\sum_j q_{j+1}\mu_i^m(S_i^{q_j,q_{j+1}})-\frac{1}{c}\leq\int (g_i\circ h_i) d\mu_i^m\leq\sum_j q_{j+1}\mu_i^m(S_i^{q_j,q_{j+1}}).\]
On the other hand, the sets $S^{q_j,q_{j+1}}$ partition $M^m$ and if $\vec x\in S^{q_j,q_{j+1}}$ then, since $g\circ h(\vec x)=\lim_{\mathcal{U}}g_i\circ h_i(\vec x_i)$, $q_j\leq g\circ h(\vec x)\leq q_{j+1}$, so also
\[\sum_j q_{j+1}\mu^m(S^{q_j,q_{j+1}})-\frac{1}{c}\leq\int (g\circ h) d\mu^m\leq\sum_j q_{j+1}\mu^m(S^{q_j,q_{j+1}}).\]

Therefore we have $|\lim_{\cU}\int (g_i\circ h_i)d\mu^m_i-\int (g\circ h)d\mu^m|\leq \frac{1}{c}$.  Letting $c$ go to $\infty$, we have the desired result.
\end{proof}

We call the function $g$ as in the conclusion of the previous theorem the \emph{ultraproduct} of the family $(g_i)$.  Observe that $g(\vec x)=\lim_{\cU}g_i(\vec x_i)$ for $\vec x=[\vec x_i]_{\cU}$.  We will particularly be interested in the case that $I=\n$, in which case, using the same notation as in the statement of the previous theorem, if $\lim_{i\to \infty} \int (g\circ h_i)\ d\mu_i^m=r$, then $\int (g\circ h)\ d\mu^m=r$.

\begin{rmk}
In the previous theorem, one could drop the requirement that the family $(g_i)$ be uniformly layerwise definable and instead only require that each $g_i$ be measurable.  However, the $\M_i'$ and, consequently, $\M'$, would only be quasistructures rather than structures.
\end{rmk}

\section{Correspondence Theorems}
\label{sec:correspondence}

\subsection{The Furstenberg Correspondence Principle}
One application of AML is to give a general framework for \emph{correspondence principles} between finite structures and dynamical systems.  We will only present a proof of the original correspondence argument given by Furstenberg \cite{furstenberg77}; since then, many variations and generalizations have been produced (for instance, \cite{bergelson:MR1776759}).  The method given extends immediately to these generalizations.

\begin{df}
Let $E\subseteq\mathbb{Z}$.  The \emph{upper Banach density} of $E$, $\overline{d}(E)$, is
\[\limsup_{m-n\rightarrow\infty}\frac{|[n,m]\cap E|}{m-n}.\]
\end{df}

\begin{df}
A \emph{dynamical system} is a tuple $(Y,\B,\mu,T)$, where $(Y,\B,\mu)$ is a probability space and $T:Y\to Y$ is a bimeasurable bijection for which $\mu(A)=\mu(TA)$ for all $A\in \B$.
\end{df}

Furstenberg's original correspondence can be stated as follows:
\begin{thm}[Furstenberg]
  Let $E\subseteq \mathbb{Z}$ with positive upper Banach density be given.  Then there is a dynamical system $(Y,\mathcal{B},\mu,T)$ and a set $A\in\mathcal{B}$ with $\mu(A)=\overline{d}(E)$ such that for any finite set of integers $U$,
  \[\overline{d}(\bigcap_{i\in U} (E-i))\geq\mu(\bigcap_{i\in U}T^{-i} A).\]
\end{thm}
\begin{proof}
Let $(\epsilon_N \ : \ N\in \n)$ be an increasing sequence of positive rational numbers such that $\sup_N\epsilon_N=\overline{d}(E)$.  Then for each $N$, we may find $n,m$ such that $m-n>N$ and $\frac{|[n,m]\cap E|}{m-n}\geq\epsilon_N$.  Let $f_{n,m}:[n,m]\to [n,m]$ be the function such that $f_{n,m}(x)=x+1$ when $x<m$ and $f_{n,m}(m)=n$.  We consider the classical measured structure $([n,m],E\cap [n,m],f_{n,m})$ equipped with the normalized counting measure and consider the corresponding $\AML$ structure $\M_{n,m}$.

Let $(Y,A,T)$ be the ultraproduct of the $\M_{n,m}$'s with respect to some nonprincipal ultrafilter on $\n$.  Let $\mathcal{B}:=\sigma(\dfbl_1(Y))$ and let $\mu$ be the canonical measure on $\B$.  It is clear from \L o\'s' theorem that $(Y,\B,\mu,T)$ is a dynamical system.  Next observe that, by our construction, $\mu(A)=\overline{d}(E)$.  Now notice that, for any finite set of integers $U$, any rational $\delta<\mu(\bigcap_{i\in U}T^{-i}A)$, and any $N$, we may find $m,n$ such that $m-n>N$ and
\[\frac{|[n,m]\cap\bigcap_{i\in U}f^{-i}_{n,m}(E\cap [n,m])|}{m-n}>\delta.\]
Fix $\gamma>0$.  Then for sufficiently large $N$, we also have that
\[\frac{|[n,m]\cap\bigcap_{i\in U}(E-i)|}{m-n}>\delta-\gamma.\]
Consequently, $\overline{d}(\bigcap_{i\in U}(E-i))\geq \mu(\bigcap_{i\in U}T^iA)$.
\end{proof}

\

\subsection{The Regularity Lemma}

\begin{df}
\label{def:bnk_algebra}
Let $\M$ be an $\AML$ structure, let $A\subseteq M$ be a set, let $n$ be a positive integer, and let $I\subseteq [1,n]$ be given.  We define $\mathcal{B}^0_{n,I}(A)$ to be the Boolean algebra of subsets of $M^n$ generated by sets of the form
\[\{(x_1,\ldots,x_n)\in M^n \setmid\M\models \varphi(x_1,\ldots,x_n)\},\]
where $\varphi$ is a formula with parameters from $A$ whose free variables belong to $I$.

When $k\leq n$, we define $\mathcal{B}^0_{n,k}(A)$ to be the Boolean algebra generated by the algebras $\mathcal{B}_{n,I}^0(A)$ as $I$ ranges over subsets of $[1,n]$ of cardinality $k$.

In all cases, we drop $^0$ to indicate the $\sigma$-algebra generated by the algebra.  When $A=\emptyset$, we omit it and write $\mathcal{B}_{n,k}$.
\end{df}
These algebras were introduced, in a somewhat different context, by Tao in \cite{tao07,tao08Norm}.

\begin{df}
Suppose that $(G,E)$ is a finite (undirected) graph, $U,U'\subseteq G$ are nonempty, and $\epsilon$ is a positive real number.
\begin{enumerate}
\item We set $d(U,U'):=\frac{|E\cap (U\times U')|}{U\times U'}$.
\item We say that $U$ and $U'$ are \emph{$\epsilon$-regular} if whenever $V\subseteq U, V'\subseteq V'$ are such that $|V|\geq\epsilon|U|$ and $|V'|\geq\epsilon|U'|$, then we have
\[|d(U,U')-d(V,V')|<\epsilon.\]
\end{enumerate}
\end{df}

\begin{thm}[Szemer\'edi's Regularity Lemma, \cite{szemeredi:MR540024}]
For any $k$ and any $\epsilon>0$, there is a $K\geq k$ such that whenever $(G,E)$ is a finite graph, there is an $n\in[k,K]$ and a partition $G=U_1\cup\cdots\cup U_n$ such that, setting $B=\{(i,j)\setmid U_i\text{ and }U_j\text{ are not }\epsilon\text{-regular}\}$, we have
\[\left|\bigcup_{(i,j)\in B}U_i\times U_j\right|\leq \epsilon|G|^2.\]
\end{thm} 
\begin{proof} 
For a contradiction, suppose the conclusion fails.  Then we may find $k,\epsilon$ such that for every $K\geq k$ there is a finite graph $(G_K,E_K)$ with no partition into at least $k$ but at most $K$ components satisfying the theorem.  The partition of a graph into singleton elements consists entirely of $\epsilon$-regular pairs, so in particular $|G_K|>K$ for all $K$.  

Let $\la_0:=\{E\}$, where $E$ is a binary predicate symbol. Inductively assume that $\AML$ signatures $\la_0\subseteq \la_1\subseteq \la_2\subseteq \cdots$ have been constructed.  Then for any two $\AML$ $\la_n$-formulae $\varphi(x)$ and $\psi(x)$, at least one of which does not belong to $\bigcup_{i=1}^{n-1} \la_i$, with only the displayed free variable, we add new unary predicates $R_{\varphi,\psi}(x)$ and $S_{\varphi,\psi}(x)$ to $\la_{n+1}$.  We let $\la:=\bigcup_{n=1}^\infty \la_n$.  We define finite measured $\la$-structures $\mathfrak{G}_K$ as follows:
\begin{itemize} 
  \item The universe of $\G_K$ is $G_K$;
  \item The measures on $\G_K$ are given by the normalized counting measure;
  \item $E$ is interpreted in $\G_K$ by $E_K$; and
  \item $R_{\varphi,\psi},S_{\varphi,\psi}$ are interpreted by induction on formulas.  Having interpreted $\varphi,\psi$, we may set $U=\varphi(G_K)$ and $U'=\psi(G_K)$.  If $U,U'$ are not $\epsilon$-regular then set $R_{\varphi,\psi}^{\G_K}=S_{\varphi,\psi}^{\G_K}=\emptyset$.  Otherwise, choose $V\subseteq U,V'\subseteq U'$ witnessing the failure of $\epsilon$-regularity and set $R_{\varphi,\psi}^{\G_K}=V$ and $S_{\varphi,\psi}^{\G_K}=V'$.
\end{itemize}

Let $\mathfrak{G}$ be the ultraproduct of the $\mathfrak{G}_K$.  We work in the canonical measure space on $G^2$.  From here on, we will identify formulae $\varphi$ with their interpretation $\varphi(G)$.  Let $h:=\mathbb{E}(\chi_E\ | \ \mathcal{B}_{2,1})$, the conditional expectation of $\chi_E$ with respect to the $\sigma$-algebra $\B_{2,1}$.  Since $h$ can be approximated by simple functions, we may write
\[h=\sum_{i\leq K_0}\alpha'_i \chi_{C_i\times D_i}+h'\]
where $||h'||_{L^2}<\epsilon^4/4$ and $C_i,D_i$ are definable.  Let $U_1,U_2,\ldots,U_{n}$ be the atoms of the finite algebra generated by $\{C_i\}\cup\{D_i\}$.  Then we have
\[h=\sum_{i,j\leq n}\alpha_{i,j} \chi_{U_i\times U_j}+h''\]
with $||h''||_{L^2}\leq ||h'||_{L^2}$ and $\{U_i\}_{i\leq n}$ is a finite partition of $G$ into definable sets.  

Let $B$ be the collection of $(i,j)$ such that $R_{U_i,U_j}$ (and therefore $S_{U_i,U_j}$) are non-empty, and for each $(i,j)\in B$, define $\beta_{i,j}=\frac{\mu(E\cap (R_{U_i,U_j}\times S_{U_i,U_j}))}{\mu(R_{U_i,U_j})\mu(S_{U_i,U_j})}$.  By \L o\'s' theorem, for each $(i,j)\in B$, we have $|\beta_{i,j}-\alpha_{i,j}|\geq\epsilon$.  Let $B^+\subseteq B$ be those $(i,j)$ such that $\alpha_{i,j}-\beta_{i,j}\geq\epsilon$.  Suppose $\mu(\bigcup_{(i,j)\in B}U_i\times U_j)\geq \epsilon/2$.  We may assume $\mu(\bigcup_{(i,j)\in B^+}U_i\times U_j)\geq\epsilon/4$ (for otherwise we carry out a similar argument with $B\setminus B^+$).

Again by \L o\'s' theorem, for each $(i,j)\in B^+$, $\mu(R_{U_i,U_j})\geq \epsilon\mu(U_i)$ and $\mu(S_{U_i,U_j})\geq\epsilon\mu(U_j)$.  Set $Z=\bigcup_{(i,j)\in B^+}(R_{U_i,U_j}\times S_{U_i,U_j})$.  Note that $Z\in\mathcal{B}_{2,1}$ and
\[0=\int\chi_E\chi_Z-h\chi_Zd\mu=\int\chi_E\chi_Z-\sum_{i,j\leq n}\alpha_{i,j}\chi_{U_i\times U_j}\chi_Zd\mu+\int h'\chi_Zd\mu\]
so we have
\begin{align*} 
\left|\int h'\chi_Zd\mu\right|
&=\left|\int \sum_{i,j\leq n}\alpha_{i,j}\chi_{U_i\times U_j}\chi_Z-\chi_E\chi_Zd\mu\right|\\
&=\left|\sum_{(i,j)\in B^+}\int\alpha_{i,j}\mu(R_{U_i,U_j})\mu(S_{U_i,U_j})-\chi_E\chi_{R_{U_i,U_j}\times S_{U_i,U_j}}d\mu\right|\\
&\geq\sum_{(i,j)\in B^+}\alpha_{i,j}\mu(R_{U_i,U_j})\mu(S_{U_i,U_j})-(\alpha_{i,j}-\epsilon)\mu(R_{U_i,U_j})\mu(S_{U_i,U_j})\\
&=\sum_{(i,j)\in B^+}\epsilon\mu(R_{U_i,U_j})\mu(S_{U_i,U_j})\\
&\geq\sum_{(i,j)\in B^+}\epsilon^3\mu(U_i)\mu(U_j)\\
&\geq\epsilon^4/4.
\end{align*}
But this is a contradiction, since $\int h'\chi_Zd\mu\leq ||h'||_{L^2}||\chi_Z||_{L^2}<\epsilon^4/4\cdot 1$.  It follows that $\mu(\bigcup_{(i,j)\in B}U_i\times U_j)\geq \epsilon/2$.

Note that what we have constructed is a partition of $G$ which roughly satisfies the regularity conditions, but in the infinite model.  Our final step is to pull this partition down to give a finite model, giving a contradiction.  By choosing $K\geq n$ large enough, we may ensure that the following hold:
\begin{itemize} 
  \item Whenever $(i,j)\not\in B$, $R_{U_i,U_j}^{\G_{K}}=\emptyset$, and therefore $U_i(G_K)$ and $U_j(G_K)$ are $\epsilon$-regular,
  \item $\mu(U_i(G_K))\leq \sqrt{2}\mu(U_i(G))$ for each $i$.
\end{itemize} 
This implies that $$\mu\left(\bigcup_{(i,j)\in B}U_i(G_K)\times U_j(G_K)\right)\leq 2\mu\left(\bigcup_{(i,j)\in B}U_i(G)\times U_j(G)\right)\leq \epsilon$$ as desired.
\end{proof}

\begin{rmk}
The regularity lemma is usually stated with an additional requirement that the cardinalities of the partition pieces be nearly equal.  This could be obtained with the following changes: for each integer $n$, include new unary predicates $P_{n,1},\ldots,P_{n,n}$ in the language, and in the finite models interpret these predicates as a partition into nearly equal pieces.  Using these predicates, refine the partition $\{U_i\}$ into one where the components have nearly the same measure.  When we pull this down to the finite model, the sets will differ in size by a very small fraction (choosing parameters correctly, this fraction can be arbitrarily small, say $\epsilon^2$), so we may redistribute a small number of points to make these pieces have exactly the same size without making much change to the edge density of large subsets.
\end{rmk}

\begin{rmk}
The method in this subsection extends very naturally to hypergraphs, using the algebras $\mathcal{B}_{k,k-1}$ in place of $\mathcal{B}_{2,1}$.  An example of a related proof for hypergraphs is given below in Section \ref{sec:hypergraph}.
\end{rmk}

\section{The Downward L\"owenheim-Skolem Theorem}
\label{sec:first_order_translation}

In this section, we prove the Downward L\"owenheim-Skolem Theorem for $\AML$ using an appropriate version of the Tarski-Vaught test.

\begin{prop}[Tarski-Vaught Test]\label{TV}
Suppose that $\M$ is a substructure of $\N$.  Then $\M$ is an elementary substructure of $\N$ if and only if:
\begin{enumerate}
\item for all formulae $\varphi(x;\vec y)$ and $\vec a$ from $M$, if there is $b\in N$ with $\N\models \varphi(b,\vec a)$, then there is $c\in M$ with $\N\models \varphi(c;\vec a)$;
\item for all formulae $\varphi(x;\vec y)$ and $\vec a$ from $M$, we have $\mu(\varphi(N,\vec a)\cap M)=\mu(\varphi(N,\vec a))$ and $v(\varphi(N,\vec a)\cap M)=v(\varphi(N,\vec a))$
\end{enumerate}
\end{prop}

\begin{proof}
First suppose that $\M$ is an elementary substructure of $\N$.  Then (1) is immediate.  To prove (2), first notice that $\varphi(N,\vec a)\cap M=\varphi(M)$ by elementarity.  Suppose, towards a contradiction, that $\mu(\varphi(M,\vec a))=r$ while $\mu(\varphi(N,\vec a))=s$ with $r\not=s$.  Without loss of generality, suppose that $r<s$.  Fix $q\in \q$ with $r<q<s$.  Then $\M\models m_{x}< q.\varphi(x,\vec y)[\vec a]$ while $\N\not\models m_{x}< q.\varphi(x,\vec y)[\vec a]$, a contradiction.

We now must show that $v(\varphi(N,\vec a)\cap M)=v(\varphi(N,\vec a))$.  Set $X=\varphi(N,\vec a)$, $Y=\varphi(M,\vec a)$, and $r:=\mu(\varphi(N,\vec a))=\mu(\varphi(M,\vec a))$.  Without loss of generality, we may assume that $r\in \q$.  First suppose that $v(X)=\oplus$.  If $v(Y)=\ominus$, then $\M\models m_{\vec x}<r. \varphi(\vec x;\vec a)$ while $\N\not \models m_{\vec x}<r. \varphi(\vec x;\vec a)$; if $v(Y)=\odot$, then $\M\models m_{\vec x}\leq r. \varphi(\vec x;\vec a)$ while $\N\not \models m_{\vec x}\leq r. \varphi(\vec x;\vec a)$.  Now suppose that $v(X)=\ominus$.  If $v(Y)\not=\ominus$, then $\M\not\models m_{\vec x}<r .\varphi(\vec x;\vec a)$ while $\N\models m_{\vec x}<r. \varphi(\vec x;\vec a)$.  Finally suppose that $v(X)=\odot$.  If $v(Y)=\ominus$, then $\M\models m_{\vec x}<r. \varphi(\vec x;\vec a)$ while $\N\not\models m_{\vec x}<r. \varphi(\vec x;\vec a)$; if $v(Y)=\oplus$, then $\M\not\models m_{\vec x}\leq r. \varphi(\vec x;\vec a)$ while $\N\models m_{\vec x}\leq r. \varphi(\vec x;\vec a)$.

Conversely, suppose that (1) and (2) hold for every formula; we prove that $\M$ is an elementary substructure of $\N$ by induction on complexity of formulae.  The quantifier case is handled as usual using (1) while the measure constructor case is handled using (2).  
\end{proof}

\begin{thm}[Downward L\"owenheim-Skolem]
Suppose that $\N$ is a structure and $X$ is a subset of $N$.  Then there is an elementary substructure $\M$ of $\N$ containing $X$ with $|M|\leq \max(|X|,|\la|,\aleph_0)$.
\end{thm}

\begin{proof}
We define a sequence $(X_i)_{i<\omega}$ of subsets of $N$ as follows.  Set $X_0:=X$.  Assume now that $X_i$ has already been defined.  For each formula $\varphi(x)$ with parameters from $X_i$ that is satisfiable in $\N$, we fix a witness $b_\varphi$.  We then let $X_{i+1}$ be the closure of $X_i$ and all the $b_\varphi$'s under the function symbols in $\la$.  Observe that, since we only added measure constructors $m_{\vec x}\bowtie r$ for rational $r$, we have $|X_i|\leq \max(|X|,|\la|,\aleph_0)$ for each $i<\omega$.

We now define an elementary substructure $\M$ as follows.  We take $M:=\bigcup_{i<\omega} X_i$ and $\B_\M:=\B_\N\cap M$.  By the usual Tarski-Vaught test, the first-order part of $\M$ is an elementary substructure of the first-order part of $\N$.  This is the beginning of an inductive argument that allows one to define, for $a\in M$, $\mu(\varphi(N,a)\cap M):=\mu(\varphi(N,a))$ and $v(\varphi(N,a)\cap M):=v(\varphi(M,a))$.  Another inductive argument allows one to show that the quasistructure $\M$ is actually a structure.  Now the $\AML$ Tarski-Vaught Test (Proposition \ref{TV}) yields that $\M$ is an elementary substructure of $\N$.   
\end{proof}

\section{Model Theory and Applications}
\label{sec:applications}
\subsection{More Model Theory}
We introduce some standard (and less standard) notions from model theory.  We often discuss the case where the language is countable, $\mathfrak{M}$ is $\aleph_1$-saturated, and we consider a countable substructure, since this is the most interesting case for our purposes.  Most of these results would still hold as long as the saturation of $\mathfrak{M}$ is greater than the cardinality of either the signature or the submodels we consider.

From now on, we fix a countable signature $\la$ and a measured $\la$-structure $\M$ (although Lemmas \ref{extend}, \ref{inv}, and \ref{ind} below make no mention of the measure and thus these Lemmas go through for an arbitrary $\la$-structure).  For example, $\M$ could be a Fubini structure equipped with its canonical measures.  Also, for sake of readability, we will write $\mu$ for the measure, regardless of the cartesian power of $M^n$ under discussion.  Finally, we will often identify a formula $\varphi(x)$ with parameters from $M$ with the subset of $M^n$ that it defines; in this way, we may write $\mu(\varphi)$ to denote the measure of the definable subset of $M^n$ defined by $\varphi$.  Likewise, we may write $\neg Y$ to indicate the complement of the definable set $Y$.

\begin{df} 
Suppose $A\subseteq M$.
\begin{enumerate}
\item A \emph{partial type} of $n$-tuples (or an \emph{$n$-type}) over $A$ is a collection $p(\vec{x})$ of formulae with parameters from $A$, where $\vec{x}=(x_1,\ldots,x_n)$, such that for every $\varphi_1,\ldots,\varphi_n\in p$, there is $\vec{a}\in M^n$ such that $\M\models \varphi_i(\vec{a})$ for every $i\leq n$.
\item A partial type $p$ is said to be a \emph{type} over $A$ if for every formula $\varphi$ with suitable free variables and parameters from $A$, either $\varphi\in p$ or $\neg\varphi\in p$.
\item A \emph{global type} is a type over $M$.
\item If $\vec a$ is a sequence of elements, we write $\tp(\vec a/A)$ for the set of formulas with parameters from $A$ satisfied by $\vec a$.
\item If $\M$ is measured, then the partial type $p$ is \emph{wide} if for every $\varphi\in p$, $\mu(\varphi)>0$.
\end{enumerate}
\end{df}

\begin{rmk}
An $\la$-structure $\M$ is $\aleph_1$-saturated if and only if, whenever $A\subseteq M$ is countable and $p(\vec x)$ is a partial type over $A$, there is $a\in M^{|\vec x|}$ such that $\M\models \varphi(\vec a)$ for all $\varphi \in p$.  In this case, we write $\vec a \models p$.  
\end{rmk}

\begin{lemma}
Suppose that $A\subseteq M$ is countable.  Then for almost every $\vec a$, $\tp(\vec a/A)$ is wide.
\end{lemma}
\begin{proof} 
Any $\vec a$ such that $\tp(\vec a/A)$ is not wide must belong to some set of measure $0$ definable with parameters from $A$.  There are countably many such sets, so their union has measure $0$.
\end{proof}

\begin{lemma}
Suppose that $A\subseteq M$ is countable.  Then for almost every $(\vec a,\vec b)$, $\tp(\vec a/A\cup\{\vec b\})$ and $\tp(\vec b/A\cup\{\vec a\})$ are both wide.
\end{lemma}
\begin{proof} 
By the previous lemma, for each $\vec b$ the set $$T_{\vec b}=\{\vec a\setmid \tp(\vec a/A\cup\{\vec b\})\text{ is not wide}\}$$ has measure $0$.  Therefore $\mu(\{(\vec a,\vec b)\setmid \vec a\in T_{\vec b}\})=\int \mu(T_{\vec b})d\mu^{|\vec b|}(\vec b)=0$, so the set of $(\vec a,\vec b)$ such that $\tp(\vec a/A\cup\{\vec b\})$ is not wide has measure $0$.  By a symmetric argument, also the set of $(\vec a,\vec b)$ such that $\tp(\vec b/A\cup\{\vec a\})$ is not wide has measure $0$.  Therefore almost every $(\vec a,\vec b)$ is in neither of these sets; for such $(\vec{a},\vec{b})$, $\tp(\vec a/A\cup\{\vec b\})$ and $\tp(\vec b/A\cup\{\vec a\})$ are both wide. 
\end{proof}

\begin{df}
Suppose that $(\vec b_i \ : \ i\in I)$ is a sequence of $r$-tuples of elements of $M$, where $I$ is an initial segment of $\n$, and $A\subseteq M$.  Then $(\vec b_i)$ is a \emph{sequence of indiscernibles over $A$} if whenever $\varphi(\vec x_0,\ldots,\vec x_n)$ is a formula with parameters from $A$ and the displayed variables and $m_0<\cdots<m_n$, we have $$\M\models \varphi(\vec b_0,\ldots,\vec b_n)\Leftrightarrow \varphi(\vec b_{m_0},\ldots,\vec b_{m_n}).$$
\end{df}
That is, for every $n$ there is a type $p_n$ of $r\cdot(n+1)$-tuples such that $\tp(\vec b_{m_0},\ldots,\vec b_{m_n}/A)=p_n$ whenever $m_0<\cdots<m_n$ is an increasing sequence.

\subsection{Amalgamation}

The arguments in this section are essentially derived from \cite{hrushovski}.  



\begin{lemma}
Suppose that $\mu(M)$ is finite, $(\vec b_i)$ is a sequence of indiscernibles over $A\subseteq M$, $X$ is a definable set over $A$, and $\mu(X_{\vec b_0})>0$ holds.  Then for any $n$, $\mu(\bigcap_{i\leq n}X_{\vec b_i})>0$.
\end{lemma}
\begin{proof}
Suppose the conclusion of the lemma is false and take $k>0$ minimal so that $\mu(\bigcap_{i\leq k}X_{\vec b_i})=0$.  For $n\geq k-1$, let $C_n=\bigcap_{i<k-1}X_{\vec b_i}\cap X_{\vec b_n}$.  Then $\mu(C_n)=\mu(\bigcap_{i\leq k-1}X_{\vec b_i})>0$ by indiscernibility.  However, when $k-1\leq n<m$, by indiscernibility again, $\mu(C_n\cap C_m)=\mu(\bigcap_{i\leq k}X_{\vec b_i})=0$.  This is a contradiction since $\mu$ is a finite measure.
\end{proof}

\begin{lemma}
  Suppose that $\M$ is $\aleph_1$-saturated, $\mu(M)$ is finite, and $A\subseteq M$ is countable.  If  $(\vec b_i)$ is a sequence of indiscernibles over $A$ and $\tp(\vec a/A\cup\{\vec b_0\})$ is wide, then for any $n$, there is an $\vec a'$ with $\tp(\vec a',\vec b_i/A)=\tp(\vec a,\vec b_0/A)$ for all $i\leq n$.
  \label{amalgamation:repeating}
\end{lemma}
\begin{proof}
  Let $X$ be a set definable over $A$ and suppose $(\vec a,\vec b_0)\in X$.  Then $\vec a\in X_{\vec b_0}$ and, since $\tp(\vec a/A\cup \{\vec b_0\})$ is wide, $\mu(X_{\vec b_0})>0$.  By the previous lemma, $\mu(\bigcap_{i\leq n}X_{\vec b_i})>0$, and is therefore non-empty.  By $\aleph_1$-saturation, there is an $a'$ such that $a'\in \bigcap_{i\leq n}X_{\vec b_i}$ simultaneously for all $X\ni (\vec a,\vec b_0)$ that are definable over $A$.  Therefore $\tp(\vec a',\vec b_i/A)=\tp(\vec a,\vec b_0/A)$ for all $i\leq n$.
\end{proof}

\begin{rmk}
The previous lemma suggests a connection between wideness and the model-theoretic notion of  \emph{nonforking}, as we now explain.  First, a formula $\varphi(x,b)$ \emph{$k$-divides} over a set $A\subseteq M$ if there is a sequence $(b_i\ : \ i\in \n)$ which is indiscernible over $A$ for which $b_0=b$ and such that $\bigcap_{i\in I} \varphi(M,b_i)=\emptyset$ for any $I\subseteq \n$ with $|I|=k$.  We say that $\varphi(x,b)$ \emph{divides} over $A$ if it $k$-divides over $A$ for some $k\geq 2$.  We say that $\varphi(x,b)$ \emph{forks} over $A$ if there are formulae $\psi_1(x,b_1),\ldots,\psi_n(x,b_n)$ such that $\varphi(M,b)\subseteq \bigcup_{i=1}^n \psi_i(M,b_i)$ and such that each $\psi_i(x,b_i)$ divides over $A$.  Finally, we say that a type $p(x)$ forks over $A$ if it contains a formula which forks over $A$.  In model theory, when $\tp(a/A\cup\{b\})$ does not fork over $A$, this implies that $a$ is in some sense independent from $b$ over $A$.  For example, when working in an algebraically closed field $K$ and $k$ is a subfield of $K$, then $\tp(a/k\cup\{b\})$ does not fork over $k$ if and only if, whenever $a$ is a generic point for a Zariski closed set $V$ defined over $k$, then $a$ does not lie in any Zariski closed set $V'$ defined over $k(b)$ with $\dim(V')<\dim (V)$.  Lemma \ref{amalgamation:repeating} shows that if $A\subseteq M$ is countable and $\tp(a/A\cup \{b\})$ is wide, then $\tp(a/A\cup\{b\})$ does not fork over $A$.
\end{rmk}

\begin{df}
 Suppose that $A\subseteq M$ and $p$ is a global type.
 \begin{enumerate}
 \item $p$ is \emph{$A$-invariant} if whenever $\tp(\vec a_1,\ldots,\vec a_n/A)=\tp(\vec a'_1,\ldots,\vec a'_n/A)$, we have $X_{\vec a_1,\ldots,\vec a_n}\in p\Leftrightarrow X_{\vec a'_1,\ldots,\vec a'_n}\in p$.  
 \item $p$ is \emph{$A$-finitely satisfiable} if whenever $X\in p$, $X\cap A\neq\emptyset$.
 \end{enumerate}
\end{df}

\begin{lemma}\label{inv}
 Suppose that $p$ is a global type and $A\subseteq M$.  If $p$ is $A$-finitely satisfiable then $p$ is $A$-invariant.
\end{lemma}
\begin{proof}
  Suppose not; then there are definable $X$ and $\vec a_1,\ldots,\vec a_n,\vec a'_1,\ldots,\vec a'_n$ so that $X_{\vec a_1,\ldots,\vec a_n}\in p$ and $X_{\vec a'_1,\ldots,\vec a'_n}\not\in p$.  Since $p$ is a global type, this means $\neg X_{\vec a'_1,\ldots,\vec a'_n}\in p$, and since $p$ is closed under intersections, $X_{\vec a_1,\ldots,\vec a_n}\setminus X_{\vec a'_1,\ldots,\vec a'_n}\in p$.  But then there is an element $m\in A$ such that $m\in X_{\vec a_1,\ldots,\vec a_n}\setminus X_{\vec a'_1,\ldots,\vec a'_n}$, contradicting the fact that $\tp(\vec a_1,\ldots,\vec a_n/A)=\tp(\vec a'_1,\ldots,\vec a'_n/A)$.
\end{proof}

\begin{lemma}\label{extend}
  Let $\N$ be an elementary substructure of $\M$.  If $p$ is a type over $N$, then there is an extension of $p$ to a global $N$-finitely satisfiable (and therefore $N$-invariant) type.
\end{lemma}
\begin{proof}
  Note that, since $\N$ is an elementary submodel of $\M$, each $X\in p$ has the property that $X\cap N^n\neq\emptyset$.  Let $$p'=p\cup\{\neg Y\setmid Y\in \dfbl_n(M) \text{ and }Y\cap N^n=\emptyset\}.$$  We will show that every finite subset of $p'$ is satisfied by an element of $N$.  If not, there is an $X\in p$ and $Y_1,\ldots,Y_n$ with $Y_i\cap N=\emptyset$ such that $X\cap\bigcap_i\neg Y_i=\emptyset$.  So $X\cap N\cap \bigcap_i(Y_i\cup\neg Y_i)\neq\emptyset$, and since each $Y_i\cap N=\emptyset$, $X\cap M\cap\bigcap_iY_i\neq\emptyset$.

  Let $q$ be an arbitrary extension of $p'$ to a (global) type.  Suppose $q$ is not finitely satisfiable in $M$; then there is a $Y\in q$ with $Y\cap N=\emptyset$, which is impossible since then $\neg Y\in p'\subseteq q$.
\end{proof}

\begin{df}
  If $S$ is a set and $p$ a global type, write $p\mid S$ for the restriction of $p$ to sets definable over $S$.
\end{df}

\begin{lemma}\label{ind}
  Suppose $p$ is a global $A$-invariant type, and recursively choose $\vec b_n\vDash p\mid A\cup\{\vec b_i\}_{i<n}$.  Then $(\vec b_i)$ (if it exists) is a sequence of indiscernibles over $A$.  (This sequence always exists if $\mathfrak{M}$ is $\aleph_1$-saturated.)
\end{lemma}
\begin{proof}
 By induction on $n$, we show $\tp(\vec b_{m_0},\ldots,\vec b_{m_n}/A)$ is constant whenever $m_0<\cdots<m_n$.  When $n=1$, this follows since each $\tp(\vec b_i/A)=p\mid A$.  Suppose now that $(\vec b_{n+1},\vec b_n,\ldots,\vec b_0)\in X$.  Then also $(\vec b_{m_{n+1}},\vec b_n,\ldots,\vec b_0)\in X$ holds because $\tp(\vec b_{m_{n+1}}/A\cup\{\vec b_0,\ldots,\vec b_n\})=\tp(\vec b_{n+1}/A\cup\{\vec b_0,\ldots,\vec b_n\})=p\mid A\cup\{\vec b_0,\ldots,\vec b_n\}$.  By the induction hypothesis, we have that $\tp(\vec b_0,\ldots,\vec b_n/A)=\tp(\vec b_{m_0},\ldots,\vec b_{m_n}/A)$.  Since $X_{\vec b_n,\ldots,\vec b_0}\in p$, also $X_{\vec b_{m_n},\ldots,\vec b_{m_0}}\in p$, whence we have $(\vec b_{m_{n+1}},\vec b_{m_n},\ldots,\vec b_{m_0})\in X$.
\end{proof}

\begin{lemma}
Suppose that $\mathfrak{M}$ is $\aleph_1$-saturated and $\mu(M)$ is finite.  Suppose that $A\subseteq M$ is countable, $q$ is an $A$-invariant type, $\vec b\vDash q\mid A\cup\{\vec a\}$, $\tp(\vec b'/A)=\tp(\vec b/A)$, $\tp(\vec a'/A)=\tp(\vec a/A)$, $\tp(\vec a'/A\cup\{\vec b'\})$ is wide, $X,Y$ are definable with parameters from $A$, and $\mu(X_{\vec a}\cap Y_{\vec b})>0$.  Then $\mu(X_{\vec a'}\cap Y_{\vec b'})>0$.
\label{amalgamation:bothbig}
\end{lemma}
\begin{proof}
Suppose the claim fails, so $\mu(X_{\vec a'}\cap Y_{\vec b'})=0$.  By $\aleph_1$-saturation and the fact that $\tp(\vec b'/A)=\tp(\vec b/A)$, there an $\vec a''$ with $\tp(\vec a'',\vec b/A)=\tp(\vec a',\vec b'/A)$; in this way, we may assume $\vec b=\vec b'$.  

Since $q$ is $A$-invariant, whenever $\vec a^*\vDash \tp(\vec a/A)$ and $\vec b^*\vDash q\mid A\cup\{\vec a^*\}$ we have $tp(\vec a^*,\vec b^*/A)=tp(\vec a,\vec b/A)$, and so in particular $\mu(X_{\vec a^*}\cap Y_{\vec b^*})=\mu(X_{\vec a}\cap Y_{\vec b})>0$.



Choose $\vec a_0\vDash p$ and $\vec b_0\vDash q\mid A\cup\{\vec a_0\}$.  Let $\{\vec a_i,\vec b_i\}_{i\leq n}$ be given so that $(\vec b_i)$ is indiscernible over $A$.  For each $\delta>0$ and each $C\in \tp(\vec a/A)$, by the wideness of $\tp(\vec a'/A\cup\{\vec b\})$, we have
\[\mu(C\cap\{\vec x\setmid \M\vDash \left(m_{\vec z}<\delta. X(\vec z,\vec x)\wedge Y(\vec z,\vec b)\right)\})>0,\]
and therefore by Lemma \ref{amalgamation:repeating} and $\aleph_1$-saturation we may find an $\vec a_{n+1}$ with $\tp(\vec a_{n+1}/A)=\tp(\vec a/A)$ such that $\mu(X_{\vec a_{n+1}}\cap Y_{\vec b_i})=0$ for $i\leq n$.  Let $\vec b_{n+1}\vDash q\upharpoonright (A\cup\{\vec a_i,\vec b_i\}_{i\leq n}\cup\{\vec a_{n+1}\})$; by Lemma \ref{ind}, we have ensured that $(\vec b_i)$ remains indiscernible over $A$.  For each $i$, let $C_i=X_{\vec a_i}\cap Y_{\vec b_i}$.  Since each $\vec b_n\vDash q\upharpoonright \vec a_n$, $\mu(C_i)\geq\epsilon$; while for $i\neq j$, $\mu(C_i\cap C_j)=0$.  This is a contradiction to the finiteness of $\mu(M)$.
\end{proof}

\begin{lemma}
Suppose that $\mathfrak{M}$ is $\aleph_1$-saturated and $\mu(M)$ is finite.  Suppose that $A\subseteq M$ is countable, $q$ is an $A$-invariant type, $\vec b\vDash q\mid A\cup\{\vec a\}$, $\tp(\vec b'/A)=\tp(\vec b/A)$, $\tp(\vec a'/A)=\tp(\vec a/A)$, $\tp(\vec a'/A\cup\{\vec b'\})$ is wide, $X,Y$ are definable with parameters from $A$, and  and $\mu(X_{\vec a}\cap Y_{\vec b})=0$.  Then $\mu(X_{\vec a'}\cap Y_{\vec b'})=0$.
\label{amalgamation:bothsmall}
\end{lemma}
\begin{proof}
Suppose the claim fails, so $\mu(X_{\vec a'}\cap Y_{\vec b'})>0$.  As in the proof of the previous lemma, we may assume that $\vec b=\vec b'$.  Suppose that $\{\vec a_i,\vec b_i\}_{i<n}$ has been constructed (including the empty case where $n=0$) so that $\vec a_i\models \tp(\vec a/A)$ for each $i$ and so that $(\vec b_i)$ is indiscernible over $A$.  Choose $\vec b_n\vDash q\mid A\cup\{\vec a_i,\vec b_i\}_{i<n}$.  There is some $\epsilon\in \q^{>0}$ such that for each $C\in \tp(\vec a/A)$,
\[\mu(C\cap\{\vec x\setmid \M\vDash \left(m_{\vec z}\geq\epsilon. X(\vec z,\vec x)\wedge Y(\vec z,\vec b)\right)\})>0,\]
 and therefore by Lemma \ref{amalgamation:repeating} and $\aleph_1$-saturation we may find an $\vec a_{n}$ with $\tp(\vec a_{n}/A)=\tp(\vec a/A)$ such that $\mu(X_{\vec a_{n}}\cap Y_{\vec b_i})\geq\epsilon$ for $i\leq n$.  
 For each $i$, let $C_i=X_{\vec a_i}\cap Y_{\vec b_i}$.  Then $\mu(C_i)\geq\epsilon$ for each $i$ and, when $i<j$, $\mu(C_i\cap C_j)=0$ since $\vec b_j\vDash q\mid A\cup\{\vec a_i\}$.  Thus $\mu(X_{\vec a_i}\cap Y_{\vec b_j})=0$ by $A$-invariance.  This contradicts the fact that $\mu(M)$ is finite.
\end{proof}

\begin{thm}
  Suppose $\M$ is $\aleph_1$-saturated and $\mu(M)$ is finite.  Let $\N$ be a countable elementary substructure of $\M$.  Let $p,q$ be types over $N$ and suppose $\tp(\vec a/N)=\tp(\vec a'/N)=p$, $\tp(\vec b/N)=\tp(\vec b'/N)=q$ with $\tp(\vec a/N\cup\{\vec b\}), \tp(\vec a'/N\cup\{\vec b'\})$ wide.  Then for any definable sets $X$ and $Y$ over $N$, we have $\mu(X_{\vec a}\cap Y_{\vec b})>0$ iff $\mu(X_{\vec a'}\cap Y_{\vec b'})>0$.
\label{amalgamation:2.16}
\end{thm}
\begin{proof}
Fix an extension $q'$ of $q$ to an $N$-invariant global type; this is possible by Lemma \ref{extend}.  Let $\vec  a^*\vDash p$ and $\vec  b^*\vDash q'\mid N\cup\{\vec  a^*\}$.  If $\mu(X_{\vec  a^*}\cap Y_{\vec  b^*})>0$ then by Lemma \ref{amalgamation:bothbig}, both $\mu(X_{\vec a}\cap Y_{\vec b})>0$ and $\mu(X_{\vec a'}\cap Y_{\vec b'})>0$.  Otherwise, $\mu(X_{\vec  a^*}\cap Y_{\vec  b^*})=0$, and by Lemma \ref{amalgamation:bothsmall}, both $\mu(X_{\vec a}\cap Y_{\vec b})=0$ and $\mu(X_{a'}\cap Y_{b'})=0$.
\end{proof}

\begin{df}
Let $\N$ be an elementary substructure of $\M$.  A set $R\subseteq M^{2m}$ is \emph{wide-stable} over $\N$ if $R$ is $N$-invariant and whenever $\tp(\vec a/N)=\tp(\vec a'/N)$, $\tp(\vec b/N)=\tp(\vec b'/N)$, both $\tp(\vec a/N\cup\{\vec b\})$ and $\tp(\vec a'/N\cup\{\vec b'\})$ are wide, and $(\vec a,\vec b)\in R$, then also $(\vec a',\vec b')\in R$.

$R$ is \emph{stable} over $\N$ if whenever $(\vec a_i,\vec b_i)$ is a sequence of indiscernibles over $N$ and $R(\vec a_i,\vec b_j)$ holds for every $i<j$, also $R(\vec a_j,\vec b_i)$ for some $i<j$.

We say $R\in \B_2(N)$ is \emph{approximated by definable stable sets (ADS)} (respectively, \emph{approximated by definable wide-stable sets (ADWS)}) over $N$ if for every $\epsilon>0$, there is a stable (respectively, wide-stable) set $S$ definable over $N$ such that $\mu(R\triangle S)<\epsilon$.
\end{df}

Observe that the wide-stable subsets of $M^{2m}$ form a $\sigma$-algebra.  It can be shown that stable subsets of $M^{2m}$ form a Boolean algebra.  Note that if we replace wideness in the definition of wide-stability with non-forking, we obtain a standard consequence of stability (see Lemma 3.3 of \cite{kimpillay}).  In particular, since wide types are non-forking, stable implies wide-stable.

\begin{thm}
Let $\mathfrak{M}$ be $\aleph_1$-saturated, let $\N$ be a countable elementary substructure of $\M$, and suppose $\mu$ is finite.  If $R\subseteq M^2$ is in $\B_2(N)$, then the following are equivalent:
 \begin{enumerate}
  \item There is $U\in\mathcal{B}_{2,1}(N)$ such that $\mu(R\triangle U)=0$;
  \item $R$ is ADS;
  \item $R$ is ADWS.
\end{enumerate}
\end{thm}
\begin{proof}
$(1)\Rightarrow (2)$:  First observe that every set of the form $A\times B$ with $A,B$ definable over $N$, is stable:  if $(a_i,b_i\setmid i\in I)$ is a sequence of indiscernibles over $N$ and $(a_i,b_j)\in A\times B$ for $i<j$ then $(a_j,b_i)\in A\times B$ for all $i<j$.  Consequently, the algebra generated by the sets of the form $A\times B$, with $A,B$ definable over $N$, which is dense in $\B_{2,1}(N)$, is contained in the algebra of stable subsets of $M^2$.  
Thus, if there is $U\in \B_{2,1}(N)$ such that $\mu(R\triangle U)=0$, then for every $\epsilon>0$, there is $S \in \B_{2,1}^0(N)$ such that $\mu(R\triangle S)<\epsilon$; such an $S$ is stable.

$(2)\Rightarrow (3)$: Immediate since stable definable sets are wide-stable and definable.

$(3)\Rightarrow (1)$: Let $R\subseteq M^2$ be definable over $N$ and wide-stable over $N$.  Since $\B_{2,1}(N)$ is complete with respect to the pseudometric $\mu(\cdot\triangle\cdot)$, it suffices to prove (1) for such an $R$.  For any $N$-definable $A$ and $B$ and any $\epsilon>0$, define $A^B_\epsilon=\{a\in A\setmid \mathfrak{M}\vDash m_y<\epsilon. B\setminus R_a\}$.  Set $$\mathcal{U}=\{A\times B\setmid \mu((A\times B)\setminus R)=0,\ A,B\in \B_1(N)\}\cup\{\bigcap_n (A^B_{1/n}\times B)\ : \ A,B\in \B_1(N)\}.$$ 
Clearly $U=\bigcup\mathcal{U}$ belongs to $\mathcal{B}_{2,1}(N)$ and satisfies $\mu(U\setminus R)=0$.

Suppose, towards a contradiction, that $\mu(R\setminus U)>0$.  Then we may find an $(a,b)\in R\setminus U$ such that $\tp(a/N)$ and $\tp(b/N\cup\{a\})$ are wide.  For each $A\in \tp(a/N)$ and $B\in \tp(b/N)$, we have $A\times B\not\in\mathcal{U}$, and therefore $\mu((A\times B)\setminus R)>0$.  Therefore $p=(\tp(a/N)\times \tp(b/N))\cup\{\neg R(x,y)\}$ is a wide partial type.  Set
\[p'=p\cup\{\neg S(x,y)\setmid (\forall\epsilon\in \q^{>0}) (m_y<\epsilon. S(x,y)\in p)\}.\]
We claim this is also a wide partial type.  Towards this end, fix $A$ and $B$ definable over $N$ such that $a\in A$ and $b\in B$.  Also fix $S(x,y)$ definable over $N$ so that, for each $\epsilon\in\q^{>0}$, $m_y<\epsilon.S(x,y)\in p$.  Since $(a,b)\notin U$, we have some $\epsilon\in \q^{>0}$ such that $A^B_\epsilon\not\in \tp(a/N)$.  Set $$A'=\{a'\in A \ : \ \M\models (m_y\geq \epsilon.(B\setminus R_{a'}))\wedge (m_y<\epsilon.S(a',y))\}.$$  Observe that $A'\in \tp(a/N)$, whence $\mu(A')>0$.  It is enough to show that $\mu((A'\times B)\cap (\neg R\cap \neg S))>0$.  Suppose this is not the case.  Then 
\[\mu((A'\times B)\setminus R)=\mu((A'\times B)\cap (S\setminus R)).\]
  However,
$$\mu((A'\times B)\setminus R)=\int_{a'\in A'}\mu(B\setminus R_{a'})d\mu\geq \epsilon\mu(A')$$ while
$$\mu((A'\times B)\cap (S\setminus R))=\int_{a'\in A'}\mu(B\cap (S_{a'}\setminus R_{a'}))d\mu<\epsilon \mu(A'),$$ yielding a contradiction.  

Now take any extension of $p'$ to a wide type $q$ and suppose that $(a',b')$ realizes $q$.  It follows immediately that $tp(b'/N\cup\{a'\})$ is wide, $tp(a'/N)=tp(a/N)$, and $tp(b'/N)=tp(b/N)$.  Since $tp(b/N\cup\{a\})$ is wide and $tp(a'/N)=tp(a/N)$, whenever $B_a\in tp(b/N\cup\{a\})$, $\mu(B_{a'})>0$.  By $\aleph_1$-saturation, there is a $b''$ with $tp(b''/N\cup\{a'\})=tp(b/N\cup\{a\})$.  But now we have $(a',b')\not\in R$ and $(a',b'')\in R$, contradicting wide-stability.
\end{proof}
We could generalize these notions to stable relations on pairs of $n$-tuples, by using the two algebras $\mathcal{B}_{2n,[1,n]}$ and $\mathcal{B}_{2n,[n+1,2n]}$ (the algebras of sets of $2n$-tuples depending only on the first $n$ or only on the second $n$ coordinates, respectively); the arguments are exactly analogous, replacing the singletons with $n$-tuples.

Combining Theorem \ref{amalgamation:2.16} and the proof of the previous theorem, we see:
\begin{cor} 
  Suppose $\mathfrak{M}$ is $\aleph_1$-saturated and $\mu(M)$ is finite.  Suppose that $\N$ is a countable elementary substructure of $\N$ and that $X$ and $Y$ are definable over $N$.  Then $\{( a, b)\setmid \mu(X_{ a}\cap Y_{ b})>0\}$ differs from an element of $\mathcal{B}_{2,1}$ by a null set.
\end{cor}

\

\subsection{Hypergraph Removal}\label{sec:hypergraph}

We now give a proof of the hypergraph removal lemma \cite{frankl:MR1884430,gowers:MR2373376,nagle:MR2198495}; this proof is closely related to the infinitary proof given by Tao \cite{tao07} in a different framework.

Suppose $\mathcal{B}$ is the $\sigma$-algebra generated by some collection of formulas with parameters from an elementary substructure $\M$ of $\N$.  Two natural ways to extend $\mathcal{B}$ would be by extending the collection of formulas and by extending the set $M$.  The following lemma states that these two methods are orthogonal in a certain sense.  This lemma will be applied in the particular case of the algebras $\mathcal{B}_{n,{<}I}(M)$, $\mathcal{B}_{n,{<}I}(M\cup\{a\})$ (to be defined below), and $\mathcal{B}_{n,I}(M)$ (as defined in Definition \ref{def:bnk_algebra}), and the reader will not be mislead assuming these are the algebras used.  (Recall our notation that $\mathcal{B}^0$ is the Boolean algebra of definable sets generating the algebra $\mathcal{B}$.)

\begin{lemma}
Let $\Phi$ be a set of formulas, let $\M$ be an elementary substructure of $\N$, and let $n$ be an integer.  Suppose $\mathcal{A}^0$ is the collection of sets of the form $\varphi(N^n,\vec a)$ where $\varphi(\vec x,\vec y)\in\Phi$ and $\vec a\in M^{|\vec y|}$, $\mathcal{B}^0$ is the set of $M$-definable sets of $n$-tuples in $\N$, and $\mathcal{C}^0$ is the collection of sets of the form $\varphi(N^n,\vec a)$ where $\varphi(\vec x,\vec y)\in\Phi$ and $\vec a\in N^{|\vec y|}$.

Then for any $f\in L^2(\mathcal{B})$, $||\mathbb{E}(f\mid\mathcal{A})-\mathbb{E}(f\mid\mathcal{C})||_{L^2}=0$.
\end{lemma}
\begin{proof} \footnote{Hrushovski has pointed out that this proof is reminiscent of early proofs in the stability theory of algebraically closed fields.  The full connection between this argument and the amalgamation methods from the previous section is not completely understood.}
Suppose not.  Then setting $\epsilon:=||f-\mathbb{E}(f\mid\mathcal{A})||_{L^2}$ and $\delta:=||f-\mathbb{E}(f\mid\mathcal{C})||_{L^2}$, we must have $\delta<\epsilon$.  Since $f$ is $\mathcal{B}$-measurable, for some $\beta_1,\ldots,\beta_m$, some $\psi_1,\ldots,\psi_m$, and some $\vec b_1,\ldots,\vec b_m$ in $M$, we have $||f-\sum_{i\leq m}\beta_i\chi_{\psi_i}(\vec x,\vec b_i)||_{L^2}<(\epsilon-\delta)/4$.

Since $||f-\mathbb{E}(f\mid\mathcal{C})||_{L^2}=\delta$, there are $\alpha_1,\ldots,\alpha_k$, formulas $\varphi_1,\ldots,\varphi_k\in\Phi$, and $\vec d_1,\ldots,\vec d_k$ in $N$ such that $||f-\sum_{i\leq k}\alpha_i\chi_{\varphi_i}(\vec x,\vec d_i)||_{L^2}<\epsilon-3(\epsilon-\delta)/4$, and therefore
\[||\sum_{i\leq m}\beta_i\chi_{\psi_i}(\vec x,\vec b_i)-\sum_{i\leq k}\alpha_i\chi_{\varphi_i}(\vec x,\vec d_i)||_{L^2}<\epsilon-(\epsilon-\delta)/2.\]
This is not itself a formula, but we can construct a formula which corresponds to it.

Squaring both sides and expanding the product, we have
\begin{eqnarray*}
\left(\epsilon-(\epsilon-\delta)/2\right)^2
>&\int \sum_{i,j\leq m}\beta_i\beta_j\chi_{\psi_i}(\vec x,\vec b_i)\chi_{\psi_j}(\vec x,\vec b_j)\\
&-2\int\sum_{i\leq m,j\leq k}\beta_i\alpha_j\chi_{\psi_i}(\vec x,\vec b_i)\chi_{\varphi_j}(\vec x,\vec d_i)\\
&+\int \sum_{i,j\leq k}\alpha_i\alpha_j\chi_{\varphi_i}(\vec x,\vec d_i)\chi_{\varphi_j}(\vec x,\vec d_j).\\
\end{eqnarray*}
For each $i\leq m, j\leq k$, we may choose an $r_{ij}$ such that $\int\chi_{\psi_i}(\vec x,\vec b_i)\chi_{\varphi_j}(\vec x,\vec d_i)d\mu<r_{ij}$ and for each $i,j\leq k$ we may choose an $s_{ij}$ such that $\int\chi_{\varphi_i}(\vec x,\vec d_i)\chi_{\varphi_j}(\vec x,\vec d_i)d\mu<s_{ij}$, and we may choose these so that
\[\left(\epsilon-(\epsilon-\delta)/2\right)^2>\int \sum_{i,j\leq m}\beta_i\beta_j\chi_{\psi_i}(\vec x,\vec b_i)\chi_{\psi_j}(\vec x,\vec b_j)-2\sum_{i\leq m,j\leq k}r_{ij}\beta_i\alpha_j+\sum_{i,j\leq k}s_{ij}\alpha_i\alpha_j.\]

Now let $\rho(\vec b_1,\ldots,\vec b_m,\vec d_1,\ldots,\vec d_k)$ be the formula
\[\bigwedge_{i\leq m,j\leq k}m_{\vec x}<r_{ij}. \varphi_i(\vec x,\vec b_i)\wedge \psi_j(\vec x,\vec d_j)\wedge\bigwedge_{i,j\leq k}m_{\vec x}<r_{ij}.\psi_i(\vec x,\vec d_i)\wedge\psi_j(\vec x,\vec d_j).\]
Since $\N\models \rho(\vec b_1,\ldots,\vec b_m,\vec d_1,\ldots,\vec d_k)$, we have $\N\models \exists \vec y_1,\ldots,\vec y_k\rho(\vec b_1,\ldots,\vec b_m,\vec y_1,\ldots,\vec y_k).$  By the elementarity of $\M$, there exist $\vec d'_1,\ldots,\vec d'_k$ in $M$ such that $\M\models \rho(\vec b_1,\ldots,\vec b_m,\vec d'_1,\ldots,\vec d'_k)$.  Consequently, for each $i\leq m, j\leq k$, we have $\int\chi_{\psi_i}(\vec x,\vec b_i)\chi_{\varphi_j}(\vec x,\vec d'_i)d\mu\leq r_{ij}$ and for each $i,j\leq k$, we have $\int \chi_{\varphi_i}(\vec x,\vec d'_i)\chi_{\varphi_j}(\vec x,\vec d'_i)d\mu\leq s_{ij}$.  Therefore
\[||\sum_{i\leq m}\beta_i\chi_{\psi_i}(\vec x,\vec b_i)-\sum_{i\leq k}\alpha_i\chi_{\varphi_i}(\vec x,\vec d'_i)||_{L^2}<\epsilon-(\epsilon-\delta)/2\]
and
\[||f-\sum_{i\leq k}\alpha_i\chi_{\varphi_i}(\vec x,\vec d'_i)||_{L^2}<\epsilon-(\epsilon-\delta)/4.\]
Since $\sum_{i\leq k}\alpha_i\chi_{\varphi_i}(\vec x,\vec d'_i)$ is measurable with respect to $\mathcal{A}$, this contradicts the assumption that $||f-E(f\mid\mathcal{A})||_{L^2}=\epsilon$.
\end{proof}

The following theorem is essentially the infinitary version of hypergraph removal:
\begin{thm}\label{counting}
Let $\mathfrak{M}$ be an elementary substructure of $\mathfrak{N}$.  Let $n\geq k$, $\mathcal{I}\subseteq{[1,n]\choose k}$, and suppose that for each each $I\in\mathcal{I}$ with $|I|=k$, we have a set $A_I\in\mathcal{B}_{n,I}(M)$, and suppose there is a $\delta>0$ such that whenever $B_I\in\mathcal{B}^0_{n,I}(M)$ and $\mu^n(A_I\setminus B_I)<\delta$ for all $I\in\mathcal{I}$, $\bigcap_{I\in\mathcal{I}} B_I$ is non-empty.  Then $\mu^n(\bigcap_{I\in\mathcal{I}} A_I)>0$.
\end{thm}
\begin{proof}
We proceed by main induction on $k$.  When $k=1$, the claim is trivial: we must have $\mu(A_I)>0$ for all $I$, since otherwise we could take $B_I=\emptyset$; then $\mu^n(\bigcap A_I)=\prod\mu(A_I)>0$.  So we assume that $k>1$ and that whenever $B_I\in\mathcal{B}_{n,I}(M)$ and $\mu^n(A_I\setminus B_I)<\delta$ for all $I$, $\bigcap_{I\in\mathcal{I}} B_I$ is non-empty.  Throughout this proof, the variables $I$ and $I_0$ range over elements of $\mathcal{I}$.

We write $\mathcal{B}_{n,{<}I}(M)$ for $\bigcup_{J\subsetneq I}\mathcal{B}_{n,J}(M)$.
\begin{claim}
For any $I_0$,
\[\int (\chi_{A_{I_0}}-E(\chi_{A_{I_0}}\mid\mathcal{B}_{n,{<}I_0}(M)))\prod_{I\neq I_0}\chi_{A_I}d\mu^{n}=0.\]
\end{claim}
\begin{claimproof}
When $k=n$, this is trivial since $\prod_{I\neq I_0}\chi_{A_I}$ is an empty product, and therefore equal to $1$.

If $k<n$, we have
\begin{align*}
 &\int (\chi_{A_{I_0}}-E(\chi_{A_{I_0}}\mid\mathcal{B}_{n,{<}I_0}^\sigma(M)))\prod_{I\neq I_0}\chi_{A_I}d\mu^{n}\\
 =&\int (\chi_{A_{I_0}}-E(\chi_{A_{I_0}}\mid\mathcal{B}_{n,{<}I_0}^\sigma(M)))\prod_{I\neq I_0}\chi_{A_I}d\mu^{k}(\{x_i\}_{i\in I_0})d\mu^{n-k}(\{x_i\}_{i\not\in I_0}).\\
\end{align*}
Observe that for any choice of $\{a_i\}_{i\not\in I_0}$, $\mathcal{B}_{n,{<}I_0}(M),\mathcal{B}_{n,{<}I_0}(M\cup\{a_i\}_{i\not\in I_0}),\mathcal{B}_{n,I_0}(M)$ satisfy the preceding lemma, so
\[||E(\chi_{A_{I_0}}\mid\mathcal{B}_{n,{<}I_0}^\sigma(M))-E(\chi_{A_{I_0}}\mid\mathcal{B}_{n,{<}I_0}^\sigma(M\cup\{a_i\}_{i\not\in I_0}))||_{L^2}=0.\]
The function $\prod_{I\neq I_0}\chi_{A_i}(\{x_i\}_{i\in I_0},\{a_i\}_{i\not\in I_0})$ is measurable with respect to $\mathcal{B}_{n,{<}I_0}^\sigma(M\cup\{\vec a_i\})$.  Combining these two facts, we have
\begin{align*}
& \int (\chi_{A_{I_0}}-E(\chi_{A_{I_0}}\mid\mathcal{B}_{n,{<}I_0}^\sigma(M)))\prod_{I\neq I_0}\chi_{A_I}d\mu^{k}(\{x_i\}_{i\in I_0})\\
=&\int (\chi_{A_{I_0}}-E(\chi_{A_{I_0}}\mid\mathcal{B}_{n,{<}I_0}^\sigma(M\cup\{a_i\}_{i\not\in I_0})))\prod_{I\neq I_0}\chi_{A_i}d\mu^k(\{x_i\})\\
=&0
\end{align*}
Since this holds for any $\{a_i\}_{i\not\in I_0}$, the claim follows by integrating over all choices of $\{a_i\}$.
\end{claimproof}

We next show that, without loss of generality, we may assume each $A_{I_0}$ belongs to $\mathcal{B}_{n,{<}I_0}(M)$ by showing that for each $I_0\in\mathcal{I}$, there is some set $A'_{I_0}\in\mathcal{B}_{n,{<}I}(M)$ with the property that, if we replace $A_{I_0}$ by $A'_{I_0}$, the assumptions of the theorem all hold, and such that if we show the conclusion for the modified family of sets, the conclusion also holds for the original family.

\begin{claim}
For any $I_0$, there is an $A'_{I_0}\in\mathcal{B}_{n,{<}I_0}(M)$ such that:
\begin{itemize}
  \item Whenever $B_I\in\mathcal{B}^0_{n,I}(M)$ for each $I$, $\mu^n(A_I\setminus B_I)<\delta$ for each $I\neq I_0$, and $\mu^n(A'_{I_0}\setminus B_{I_0})<\delta$, $\bigcap_{I\in\mathcal{I}}B_I$ is non-empty, and
  \item If $\mu^n(A'_{I_0}\cap \bigcap_{I\neq I_0}A_I)>0$, $\mu^n(\bigcap_{I\in\mathcal{I}} A_I)>0$.
\end{itemize}
\end{claim}
\begin{claimproof}
 Define $A'_{I_0}:=\{x_{I_0}\mid \mathbb{E}(\chi_{A_{I_0}}\mid\mathcal{B}_{n,<I_0}(M))(x_{I_0})>0\}$.  If $\mu^n(A'_{I_0}\cap\bigcap_{I\neq I_0}A_I)>0$ then we have
 \[\int \mathbb{E}(\chi_{A_{I_0}}\mid\mathcal{B}_{n,{<}I_0}(M))\prod_{I\neq I_0}\chi_{A_I}\,d\mu^n>0,\]
and by the previous claim, this implies that that $\mu^n(\bigcap A_I)>0$.
 
 Suppose that for each $I$, $B_I\in\mathcal{B}^0_{n,I}(M)$ with $\mu^n(A_I\setminus B_I)<\delta$ for $I\neq I_0$ and $\mu^n(A'_{I_0}\setminus B_{I_0})<\delta$.  Since
\[\mu^n(A_{I_0}\setminus A'_{I_0})=\int\chi_{A_{I_0}}(1-\chi_{A'_{I_0}})\,d\mu^n=\int \mathbb{E}(\chi_{A_{I_0}}\mid\mathcal{B}_{n,<I_0}(M))(1-\chi_{A'_{I_0}})\,d\mu^n=0,\]
we have $\mu^n(A_{I_0}\setminus B_{I_0})<\delta$ as well, and therefore $\bigcap_{I\in\mathcal{I}} B_I$ is non-empty.
\end{claimproof}

By applying the previous claim to each $I\in \mathcal{I}$, we may assume for the rest of the proof that for each $I$, $A_I\in\mathcal{B}_{n,<I}(M)$.

Fix some finite algebra $\mathcal{B}\subseteq\mathcal{B}^0_{n,k-1}(M)$ so that for every $I$, $\|\chi_{A_I}-\mathbb{E}(\chi_{A_I}\mid\mathcal{B})\|_{L^2(\mu^n)}<\frac{\sqrt{\delta}}{\sqrt{2}(|\mathcal{I}|+1)}$ (such a $\mathcal{B}$ exists because there are finitely many $I$ and each $A_I$ is $\mathcal{B}_{n,k-1}(M)$-measurable).  For each $I$, set $A^*_I=\{a_I\mid \mathbb{E}(\chi_{A_I}\mid\mathcal{B})(a_I)>\frac{|\mathcal{I}|}{|\mathcal{I}|+1}\}$.

\begin{claim}
For each $I$, $\mu^n(A_I\setminus A^*_I)\leq \delta/2$
\end{claim}
\begin{claimproof}
$A_I\setminus A^*_I$ is the set of points such that $\left(\chi_{A_I}-\mathbb{E}(\chi_{A_I}\mid\mathcal{B})\right)(\vec a)\geq\frac{1}{|\mathcal{I}|+1}$.  By Chebyshev's inequality, the measure of this set is at most
\[(|\mathcal{I}|+1)^2\int (\chi_{A_I}-\mathbb{E}(\chi_{A_I}\mid\mathcal{B}))^2\,d\mu^n=(|\mathcal{I}|+1)^2\|\chi_{A_I}-\mathbb{E}(\chi_{A_I}\mid\mathcal{B})\|_{L^2(\mu^n)}^2\leq\frac{\delta}{2}.\]
\end{claimproof}

\begin{claim}
$\mu^n(\bigcap_I A_I)\geq\mu^n(\bigcap_I A^*_I)/\left(|\mathcal{I}|+1\right)$. 
\end{claim}
\begin{claimproof}
 For each $I_0$, 
 \begin{align*}
 \mu^n((A^*_{I_0}\setminus A_{I_0})\cap\bigcap_{I\neq I_0} A^*_I)
 &=\int \chi_{A^*_{I_0}}(1-\chi_{A_{I_0}})\prod_{I\neq I_0}\chi_{A^*_I}\,d\mu^n\\
 &=\int\chi_{A^*_{I_0}}(1-\mathbb{E}(\chi_{A_{I_0}}\mid\mathcal{B}))\prod_{I\neq I_0}\chi_{A^*_I}\,d\mu^n\\
 &\leq \frac{1}{|\mathcal{I}|+1}\int\prod_{I\in\mathcal{I}}\chi_{A^*_I}\,d\mu^n\\
 &=\frac{1}{|\mathcal{I}|+1}\mu^n(\bigcap_{I\in\mathcal{I}} A^*_I)
\end{align*}

But then
\[\mu^n(\bigcap_{I\in\mathcal{I}} A^*_I\setminus \bigcap_{I\in\mathcal{I}} A_I)
\leq \sum_{I_0}\mu^n((A^*_{I_0}\setminus A_{I_0})\cap\bigcap_{I\neq I_0} A^*_I)
\leq \frac{|\mathcal{I}|}{|\mathcal{I}|+1}\mu^n(\bigcap_{I\in\mathcal{I}} A^*_I).\]
\end{claimproof}

Each $A^*_I$ may be written in the form $\bigcup_{i\leq r_I}A^*_{I,i}$ where $A^*_{I,i}=\bigcap_{J\in {I\choose k-1}}A^*_{I,i,J}$ and $A^*_{I,i,J}$ is an element of $\mathcal{B}^0_{n,J}(M)$.  We may assume that if $i\neq i'$ then $A^*_{I,i}\cap A^*_{I,i'}=\emptyset$.

We have
\[\mu^n(\bigcap_I A^*_I)=\mu^n(\bigcup_{\vec i\in\prod_I [1,r_I]}\bigcap_I\bigcap_{J\in{I\choose k-1}}A^*_{I,i_I,J}).\]
For each $\vec i\in\prod_I [1,r_I]$, let $D_{\vec i}=\bigcap_I\bigcap_{J\in{I\choose k-1}}A^*_{i_I,J,I}$.  Each $A^*_{I,i_I,J}$ is an element of $\mathcal{B}^0_{n,J}(M)$, so we may group the components and write $D_{\vec i}=\bigcap_{J\in{[1,n]\choose k-1}}D_{\vec i,J}$ where $D_{\vec i,J}=\bigcap_{I\supset J}A^*_{I,i_I,J}$.

Suppose, for a contradiction, that $\mu^n(\bigcap_I A^*_I)=0$.  Then for every $\vec i\in\prod_I [1,r_I]$, $\mu^n(D_{\vec i})=\mu^n(\bigcap_J D_{\vec i,J})=0$.  By the inductive hypothesis, for each $\gamma>0$, there is a collection $B_{\vec i,J}\in\mathcal{B}^0_{n,J}(M)$ such that $\mu^n(D_{\vec i,J}\setminus B_{\vec i,J})<\gamma$ and $\bigcap_J B_{\vec i,J}=\emptyset$.  In particular, this holds with $\gamma=\frac{\delta}{2{k\choose k-1}(\prod_I r_I)(\max_I r_I)}$.

For each $I,i\leq r_I, J\subset I$, define
 \[B^*_{I,i,J}=A^*_{I,i,J}\cap\bigcap_{\vec i, i_I=i}\left[B_{\vec i,J}\cup\bigcup_{I'\supseteq J, I'\neq I}\overline{A^*_{I',i_{I'},J}}\right].\]
 
 \begin{claim}
 $\mu^n(A^*_{I,i,J}\setminus B^*_{I,i,J})\leq \frac{\delta}{2{k\choose k-1}(\max_I r_I)}.$ 
\end{claim}
\begin{claimproof} 
Observe that if $x\in A^*_{I,i,J}\setminus B^*_{I,i,J}$ then for some $\vec i$ with $i_I=i$, $x\not\in B_{\vec i,J}\cup\bigcup_{I'\supseteq J, I'\neq I}\overline{A^*_{I',i_{I'},J}}$.  This means $x\not\in B_{\vec i,J}$ and $x\in\bigcap_{I'\supseteq J}A^*_{I',i_{I'},J}=D_{\vec i,J}$.  So
\[\mu^n(A^*_{I,i,J}\setminus B^*_{I,i,J})\leq \sum_{\vec i\in\prod_I [1,r_I]}\mu^n(D_{\vec i,J}\setminus B_{\vec i,J})\leq \frac{\delta}{2{k\choose k-1}(\max_I r_I)}.\]
\end{claimproof}

 Define $B^*_{I}=\bigcup_{i\leq r_I}\bigcap_J B^*_{I,i,J}$.
 
 \begin{claim}
 $\mu^n(A_I\setminus B^*_{I})\leq\delta$. 
\end{claim}
\begin{claimproof}
Since $\mu^n(A_I\setminus A^*_I)\leq\delta/2$, it suffices to show that $\mu^n(A^*_I\setminus B^*_{I})\leq\delta/2$.

 \begin{align*} 
 \mu^n(A^*_I\setminus \bigcup_i\bigcap_J B^*_{I,i,J})
&=\mu^n\left(\bigcup_{i}\bigcap_{J}A^*_{I,i,J}\setminus\bigcup_i\bigcap_J B^*_{I,i,J}\right)\\
&\leq\mu^n\left(\bigcup_{i}\left(\bigcap_{J}A^*_{I,i,J}\setminus\bigcap_J B^*_{I,i,J}\right)\right)\\
&\leq\sum_{i\leq r_I}\mu^n\left(\bigcap_J A^*_{I,i,J}\setminus\bigcap_J B^*_{I,i,J}\right)\\
&\leq\sum_{i\leq r_I}\sum_{J}\mu^n(A^*_{I,i,J}\setminus B^*_{I,i,J})\\
&\leq r_I\cdot {k\choose k-1}\cdot \frac{\delta}{2{k\choose k-1}(\max_I r_I)}\\
&\leq\delta/2.
\end{align*}
\end{claimproof}

Note that the sets $B_I^*$ satisfy the assumption, and therefore $\bigcap_{I}B^*_{I}\neq\emptyset$.

\begin{claim} 
\[\bigcap_{I}B^*_{I}\subseteq\bigcup_{\vec i}\bigcap_J B_{\vec i,J}.\]
\end{claim}
\begin{claimproof} 
Suppose $x\in\bigcap_{I}B^*_{I}=\bigcap_I\bigcup_{i\leq r_I}\bigcap_J B^*_{I,i,J}$.  Then for each $I$, there is an $i_I\leq r_I$ such that $x\in\bigcap_J B^*_{I,i_I,J}$.  Since $B^*_{I,i_I,J}\subseteq A^*_{I,i_I,J}$, for each $I$ and $J\subset I$, $x\in A^*_{I,i_I,J}$.

For any $J$, let $I\supset J$.  Then
\[x\in B^*_{I,i_I,J}=A^*_{I,i_I,J}\cap\bigcap_{\vec i',i'_I=i_I}\left[B_{\vec i,J}\cup\bigcup_{I'\supseteq J, I'\neq I}\overline{A^*_{I',i_{I'},J}}\right].\]
In particular, $x\in \left[B_{\vec i,J}\cup\bigcup_{I'\supseteq J, I'\neq I}(\overline{A^*_{I',i_{I'},J}})\right]$ for the particular $\vec i$ we have chosen.  Since $x\in A^*_{I,i_{I'},J}$ for each $I'\supset J$, it must be that $x\in B_{\vec i,J}$.  This holds for any $J$, so $x\in \bigcap_J B_{\vec i,J}$.
\end{claimproof}

Since $\bigcap_{I}B^*_{I}$ is non-empty, there is some $\vec i$ such that $\bigcap_J B_{\vec i,J}\neq\emptyset$.  But this leads to a contradiction, so it must be that $\mu^n(\bigcap_I A^*_I)>0$, and therefore, as we have shown, $\mu^n(\bigcap_{I\in\mathcal{I}} A_I)\geq\frac{1}{|\mathcal{I}|+1}\mu^n(\bigcap_{I\in\mathcal{I}} A^*_I)>0$.
\end{proof}

\begin{cor}
  For each $\epsilon>0,k$ and each $k$-regular hypergraph $(W,F)$ there is a $\delta>0$ such that whenever $(V,E)$ is a $k$-regular hypergraph such that there are at most $\delta|V|^{|W|}$ copies of $(W,F)$, it is possible to remove at most $\epsilon|V|^k$ edges to obtain a hypergraph with no copies $(W,F)$.
\end{cor}
\begin{proof}
  Suppose not.  Fix $\epsilon>0, k, (W,F)$ so that for each $K$, there is an $k$-regular hypergraph $(V_K,E_K)$ with at most $\delta|V|^{|W|}$ copies of $(W,F)$ and such that every sub-hypergraph with at least $|E_K|-\epsilon|V_K|^n$ edges contains a copy.

We consider a signature of $\AML$ consisting of an $k$-ary predicate $E$.  We consider finite measured structures $\N_K$ defined as follows:
\begin{itemize}
\item The universe of $\N_K$ is $V_K$,
\item The measures on $\N_K$ are given by the normalized counting measure,
\item $E$ is interpreted by $E_K$.
\end{itemize}

Let $\N$ be a (nonprincipal) ultraproduct of these structures and fix a countable elementary submodel $\M$.  We may assume $W=[1,n]$.  For each $I\in{[1,n]\choose k}$, define $E^I\in\mathcal{B}_{n,I}(M)$ to be the copy of $E$ on the coordinates $I$.  Observe that if $\vec x\in N^{n}$, $\vec x$ is a copy of $W$ exactly if $W\in\bigcap_{I\in F}E^I$.  In particular, $\mu^n(\bigcap_{I\in F} E^I)=0$.

By the preceding theorem, we can find definable sets $B_I$ such that $\mu^n(E^I\setminus B_I)<\epsilon/|\mathcal{I}|$ for each $I\in{[1,n]\choose k}$ so that $\bigcap_I B_I=\emptyset$.  Note that saying this intersection of definable sets is empty is expressed by a single formula, and is therefore true for almost every structure $(V_K,E_K)$.  But then for some $K$, $(V_K,E_K\cap\bigcap B_I(G_K))$ is a sub-hypergraph obtained by removing fewer than $\epsilon|V_K|^k$ edges and containing no copies of $(W,F)$, yielding a contradiction.
\end{proof}

The second author has extended this method \cite{henry12:_analy_approac_spars_hyper} to the setting where $(V,E)$ is a dense sub-hypergraph of a sparse random hypergraph.  In this setting one replaces the normalized counting measure with a new counting measure normalized by the ambient random hypergraph, causing substantial additional complications.

\begin{rmk}
  Szemer\'edi's Theorem follows from hypergraph removal using the following, now standard, encoding.  Given $A\subseteq [1,n]$, to find an arithmetic progression of length $k+1$, we define a $k+1$-partite $k$-regular hypergraph.  For each $i\in[1,k]$, we take the part $X_i$ to be a copy of $[1,n]$, and we take $X_{k+1}$ to be a copy of $[1,k^2n]$.  Given $x_i\in X_i$ for all $i$, we say $(x_1,\ldots,x_{k})$ is an edge iff $\sum_{i\leq k}i\cdot x_i\in A$, and when $1\leq i\leq k$, $(x_1,\ldots,x_{i-1},x_{i+1},\ldots,x_{k+1})$ is an edge iff $\sum_{j\leq k,j\neq i}j\cdot x_j+i(x_{k+1}-\sum_{j\leq k,j\neq i}x_j)\in A$.

Suppose $(x_1,\ldots,x_{k+1})$ is a copy of the complete $k$-regular hypergraph on $k+1$ vertices (that is, every $k$-tuple is an edge) with $x_{k+1}\neq \sum_{i\leq k}x_i$.  Then set $a=\sum_{i\leq k}i\cdot x_i$ and $d=x_{k+1}-\sum_{i\leq k}x_i\neq 0$.  Then we have $a\in A$, and for each $i\leq k$, we have
\[a+id=\sum_{j\leq k}j\cdot x_j+i(x_{k+1}-\sum_{j\leq k}x_j)=\sum_{j\leq k,j\neq i}j\cdot x_j+i(x_{k+1}-\sum_{j\leq k,j\neq i}x_j)\in A.\]

On the other hand, for any $a\in A$ and any sequence $(x_1,\ldots,x_k)$ with $a=\sum_{i\leq k}i\cdot x_i$, the sequence $(x_1,\ldots,x_k,\sum_{i\leq k}x_i)$ is also a copy of the complete $k$-regular hypergraph on $k+1$ vertices.  It is not possible to remove all such sequences by removing a small number of edges, so hypergraph removal implies that there must be many copies of the complete $k$-regular hypergraph on $k+1$ vertices; but there are only a small number of sequences $(x_1,\ldots,x_k,\sum_{i\leq k}x_i)$, so the remaining copies must correspond to genuine arithmetic progressions.
\end{rmk}

\

\subsection{Gowers Norms}
\begin{df} 
Let $G$ be a finite abelian group and consider $g:G\rightarrow\mathbb{R}$.  We define the \emph{$k$-th Gowers uniformity norm}, $||\cdot||_{U^k}$, by
\[||g||_{U^k}^{2^k}=\frac{1}{|G|^{k+1}}\sum_{x\in G}\sum_{\vec h \in G^k}\prod_{\omega\in\{0,1\}^k}g(x+\omega\cdot \vec h).\]
\end{df}

\noindent It is not immediate that the right-hand side of the above display is nonnegative; however, this can be derived using an argument similar to that in the proof of Lemma \ref{CS} below.

The Gowers norms were introduced by Gowers in his proof of Szemer\'edi's Theorem \cite{gowers01}.  Since then, the norms have proven to be powerful tools in combinatorics; for instance, these norms have been used to give a proof of the hypergraph regularity lemma \cite{gowers:MR2373376} and in the proof of the Green-Tao theorem on arithmetic progressions in the primes \cite{green:MR2415379}.  Related norms for dynamical systems, the Gowers-Host-Kra norms, we introduced by Host and Kra \cite{host05}, and have similarly been used to prove recurrence theorems in dynamical systems (for a few examples, see \cite{bergelson05,MR2465660}; \cite{2011arXiv1103.3808F} describes the general method and cites more than twenty examples).

Roughly speaking, the $U^2$ norm is large when a function has a large correlation with some group character.  The $U^k$ norms for larger $k$ are large when a function correlates with a canonical function of a more general type.  The exact characterization of these canonical functions was a substantial project \cite{green:MR2391635,green:MR2747135}.  Here we show that, in the setting of AML, these canonical functions can be taken to be the simple functions generated by algebras $\mathcal{B}_{k,k-1}$.

\

\noindent In the infinite setting, we give a slightly more general definition:
\begin{df} 
Let $\M$ be a measured $\AML$ structure and let $f:M^k\rightarrow\mathbb{R}$ be bounded and $\mathcal{B}_k$-measurable.  Define $||\cdot||_{U^k_\infty}$ by:
\[||f||_{U^k_\infty}^{2^k}=\int \prod_{\omega\in\{0,1\}^k}f(h^{\omega(1)}_1,\ldots,h^{\omega(k)}_k) d\mu^{2k}(\vec h^0,\vec h^1).\]

\noindent Suppose further that $+$ is a definable group operation on $M$ and $g:M\rightarrow\mathbb{R}$ is bounded.  Then define $f:M^k\to \r$ by $f(h_1,\ldots,h_k)=g(\sum_{i=1}^k h_i)$ and define $||g||_{U^k_\infty}:=||f||_{U^k_\infty}$. 
\end{df}
\noindent Note that in the infinite setting, these are seminorms (and not even a seminorm for $k=1$).

\begin{lemma} 
Let $(G_i\setmid i\in \n)$ be a sequence of finite abelian groups and, for each $i$, let $g_i:G_i\rightarrow[-1,1]$ be a function such that $\lim_{i\rightarrow\infty}||g_i||_{U^k}$ exists.  Let $\mathcal{L}$ be the signature obtained by adding to the signature for abelian groups unary predicates $P_q$ for $q\in \mathbb{Q}\cap [-1,1]$.  Let $\mathfrak G_i$ be the $\AML$ $\la$-structure associated to $G_i$, equipped with its normalized counting measure, by interpreting $P_q^{\mathfrak{G_i}}:=\{a\in G_i \ : \ g_i(a)<q\}$.  Let $\cU$ be a nonprincipal ultrafilter on $\n$, let $\mathfrak G:=\prod_{\cU}\mathfrak G_i$, and let $g:G\to [-1,1]$ be the ultraproduct of $(g_i)$.  Then $||g||_{U^k_\infty}=\lim_{i\rightarrow\infty}||g_i||_{U^k}$.
\end{lemma}
\begin{proof}
First observe that 
\[||g_i||_{U^k}^{2^k}=\frac{1}{|G_i|^{2k}}\sum_{\vec h^0,\vec h^1\in G^k}\prod_{\omega\in\{0,1\}^k}g_i(\sum_i h^{\omega(i)}_i),\]
which is easily seen by making the substitutions $x=\sum_i h^0_i$ and $h_i=h^1_i-h^0_i$.  Define $j_i:G_i^k\to G_i^{2^k}$ by $j_i(\vec h)=(\sum h_p^{\omega(p)} \setmid \omega\in \{0,1\}^k)$; observe that $(j_i)$ is uniformly definable in $\mathfrak G_i$.  Now define $g_i':G^{2^k}_i\to [-1,1]$ by $g_i'((x_\omega))=\prod_{\omega \in \{0,1\}^k}g_i(x_{\omega})$.  Then the lemma follows from Theorem \ref{intultra2} (applied to $(g_i')$ and $(j_i)$).
\end{proof}

For the rest of this subsection, we fix $f:M\to \r$ which is bounded and $\B_k$-measurable.  We further assume that each $\mu^k$ is a probability measure on $\B_k$.
\begin{lemma}\label{CS}
\[\left|\int f d\mu^k\right|\leq ||f||_{U^k_\infty}.\]
\end{lemma}
\begin{proof}
Using Fubini's theorem, we have 
$$\left|\int f(\vec h)d\mu^k(\vec h)\right|^{2^k}=\left|\int \left(\int f(\vec h)d\mu(h_k)\right)d\mu^{k-1}(h_1,\ldots,h_{k-1})\right|^{2^k}.$$
By Cauchy-Schwarz and Fubini again, we have 
$$\left|\int f(\vec h)d\mu^k(\vec h)\right|^{2^k}\leq \left(\int f(h_1,\ldots,h_{k-1},h_k^0)f(h_1,\ldots,h_{k-1},h_k^1)d\mu^{k+1}\right)^{2^{k-1}}.$$ 
Repeating this process, we arrive at
\begin{align*}
\left|\int f(\vec h)d\mu^k(\vec h)\right|^{2^k}&\leq \left(\int\prod_{\omega\in\{0,1\}^k}f(h^{\omega(1)}_1,\ldots,h^{\omega(k)}_k) d\mu^{2k}(\vec h^0,\vec h^1)\right)\\
&=||f||_{U^k_\infty}^{2^k}.
\end{align*}
\end{proof}

\begin{lemma}
For each $I\subseteq[1,k]$ with $|I|=k-1$, let $B_I$ be a definable set in $\mathcal{B}_{k,I}$.  Then $0\leq ||f\prod_I \chi_{B_I}||_{U^k_\infty}\leq ||f||_{U^k_\infty}$.
\end{lemma}
\begin{proof} 
It suffices to show that $0\leq ||f\chi_{B_I}||_{U^k_\infty}\leq ||f||_{U^k_\infty}$ for a single $I$.  Without loss of generality, we assume $I=[1,k-1]$.  Since $f=f\chi_{B_I}+f\chi_{\overline{B_I}}$, we have
\[||f||^{2^k}_{U^k_\infty}=||f\chi_{B_I}+f\chi_{\overline{B_I}}||^{2^k}_{U^k_\infty}\]
which in turn expands into a sum of $2^{2^k}$ terms of the form
\[\int \prod_{\omega\in\{0,1\}^k}(f\chi_{S_\omega})(h^{\omega(1)}_1,\ldots,h^{\omega(k)}_k) d\mu^{2k}(\vec h^0,\vec h^1),\]  
\noindent where each $S_\omega$ is either $\chi_{B_I}$ or $\chi_{\overline{B_I}}$.  Observe that $\|f\chi_{B_I}\|_{U^k_\infty}^{2^k}$ corresponds to the term when each $S_\omega=\chi_{B_I}$.  Thus, it suffices to show that all of the $2^{2^k}$ terms are non-negative.

Observe that if $\omega(i)=\omega'(i)$ for each $i\in I$ but $S_\omega\neq S_{\omega'}$ then for any $\chi_{S_\omega}\chi_{S_{\omega'}}=0$, whence the corresponding integral is $0$.  We thus restrict ourselves to the case where, whenever $\omega\upharpoonright I=\omega'\upharpoonright I$, $S_\omega=S_{\omega'}$.  In this case, we have

\begin{align*}
 \int \prod_{\omega\in\{0,1\}^k}(f\chi_{S_\omega})(h^{\omega(1)}_1,\ldots,h^{\omega(k)}_k) d\mu^{2k}(\vec h^0,\vec h^1)\\
 =\int \int \prod_{\omega\in\{0,1\}^k}(f\chi_{S_\omega})(h^{\omega(1)}_1,\ldots,h^{\omega(k)}_k) d\mu^2(h^0_k,h^1_k) d\mu^{2(k-1)}\\
 =\int \left[\int \prod_{\omega\in\{0,1\}^k}(f\chi_{S_\omega})(h^{\omega(1)}_1,\ldots,h_k^{\omega(k)})d\mu(h_k)\right]^2 d\mu^{2(k-1)}.\\
 \end{align*}
Since the inside of the integral is always non-negative, this term is non-negative.
\end{proof}

We set $D(f)(\vec h^0):=\int \prod_{\omega\in\{0,1\}^k,\omega\neq\vec 0}f(h^{\omega(1)}_1,\ldots,h^{\omega(k)}_k) d\mu^k(\vec h^1)$.  Observe that $\|f\|_{U^k_\infty}^{2^k}=\int f\cdot D(f)d\mu^k(\vec h^0)$.
\begin{lemma}
$D(f)$ is measurable with respect to $\mathcal{B}_{k,k-1}$.
\end{lemma}
\begin{proof}
It suffices to show that if $||\mathbb{E}(f\ | \ \mathcal{B}_{k,k-1})||=0$ then $$\int f(\vec h^0) D(f)(\vec h^0) d\mu^k(\vec h^0)=0.$$   We have 
\begin{align*}
\int f(\vec h^0) D(f)(\vec h^0) d\mu^k(\vec h^0)
=&\int f(\vec h^0)\prod_{\omega\in\{0,1\}^k,\omega\neq\vec 0}f(h^{\omega(1)}_1,\ldots,h^{\omega(k)}_k)d\mu^{2k}(\vec h^0,\vec h^1)\\
=&\iint f(\vec h^0)\prod_{\omega\in\{0,1\}^k,\omega\neq\vec 0}f(h^{\omega(1)}_1,\ldots,h^{\omega(k)}_k)d\mu^k(\vec h^0)d\mu^k(\vec h^1).\\
\end{align*}
Observe that, for a given $\vec h^1$, we have
\[\int f(\vec h^0)\prod_{\omega\in\{0,1\}^k,\omega\neq\vec 0}f(h^{\omega(1)}_1,\ldots,h^{\omega(k)}_k)d\mu^k(\vec h^0)=0\]
since $\prod_{\omega\in\{0,1\}^k,\omega\neq\vec 0}f(h^{\omega(1)}_1,\ldots,h^{\omega(k)}_k)$ is $\mathcal{B}_{k,k-1}$-measurable.  The lemma now follows.
\end{proof}

\begin{thm} 
$||f||_{U^k_\infty}>0$ if and only if $||\mathbb{E}(f\ | \ \mathcal{B}_{k,k-1})||>0$.
\end{thm}
\begin{proof} 
If $||f||_{U^k_\infty}>0$, then $\int f D(f)d\mu^k(\vec h^0)>0$; since $D(f)$ is $\mathcal{B}_{k,k-1}$-measurable, $||\mathbb{E}(f\ | \ \mathcal{B}_{k,k-1})||>0$.

For the other direction, if $||\mathbb{E}(f\ | \ \mathcal{B}_{k-1})||>0$, we may find $\B_I$-measurable $B_I$ so that $\int f\prod\chi_{B_I}d\mu\neq 0$, whence
\[0<|\int f\prod_I \chi_{B_I}d\mu|\leq ||f\prod_I\chi_{B_I}||_{U^k_\infty}\leq ||f||_{U^k_\infty}.\]
\end{proof}

\bibliographystyle{plain}
\bibliography{AML}

\begin{thebibliography}{10}

\bibitem{austin:MR2747063}
Tim Austin.
\newblock Deducing the multidimensional {S}zemer\'edi theorem from an
  infinitary removal lemma.
\newblock {\em J. Anal. Math.}, 111:131--150, 2010.

\bibitem{MR2505436}
Seyed-Mohammad Bagheri and Massoud Pourmahdian.
\newblock The logic of integration.
\newblock {\em Arch. Math. Logic}, 48(5):465--492, 2009.

\bibitem{bergelson05}
V.~Bergelson, B.~Host, B.~Kra, and I.~Ruzsa.
\newblock Multiple recurrence and nilsequences.
\newblock {\em Inventiones Mathematicae}, 160(2):261--303, May 2005.

\bibitem{MR891243}
Vitaly Bergelson.
\newblock Ergodic {R}amsey theory.
\newblock In {\em Logic and combinatorics ({A}rcata, {C}alif., 1985)},
  volume~65 of {\em Contemp. Math.}, pages 63--87. Amer. Math. Soc.,
  Providence, RI, 1987.

\bibitem{bergelson:MR1776759}
Vitaly Bergelson.
\newblock Ergodic theory and {D}iophantine problems.
\newblock In {\em Topics in symbolic dynamics and applications ({T}emuco,
  1997)}, volume 279 of {\em London Math. Soc. Lecture Note Ser.}, pages
  167--205. Cambridge Univ. Press, Cambridge, 2000.

\bibitem{changkeisler}
C.C. Chang and J.~Keisler.
\newblock {\em Model Theory}, volume~73 of {\em Studies in Logic and the
  Foundations of Mathematics}.
\newblock Elsevier, 1990.

\bibitem{pillay:MR1650667}
Z.~Chatzidakis and A.~Pillay.
\newblock Generic structures and simple theories.
\newblock {\em Ann. Pure Appl. Logic}, 95(1-3):71--92, 1998.

\bibitem{MR2465660}
Qing Chu.
\newblock Convergence of weighted polynomial multiple ergodic averages.
\newblock {\em Proc. Amer. Math. Soc.}, 137(4):1363--1369, 2009.

\bibitem{elek08}
G.~Elek.
\newblock Weak convergence of finite graphs, integrated density of states and a
  cheeger type inequality.
\newblock {\em Journal of Combinatorial Theory, Series B}, 98(1):62 -- 68,
  2008.
\newblock 10.1016/j.jctb.2007.03.004.

\bibitem{elek07}
G.~Elek and B.~Szegedy.
\newblock Limits of hypergraphs, removal and regularity lemmas. a non-standard
  approach.
\newblock unpublished, 2007.

\bibitem{frankl:MR1884430}
Peter Frankl and Vojtech R{\"o}dl.
\newblock Extremal problems on set systems.
\newblock {\em Random Structures Algorithms}, 20(2):131--164, 2002.

\bibitem{2011arXiv1103.3808F}
N.~{Frantzikinakis}.
\newblock {Some open problems on multiple ergodic averages}.
\newblock {\em ArXiv e-prints}, March 2011.

\bibitem{furstenberg:MR1039473}
H.~Furstenberg and Y.~Katznelson.
\newblock Idempotents in compact semigroups and {R}amsey theory.
\newblock {\em Israel J. Math.}, 68(3):257--270, 1989.

\bibitem{furstenberg77}
Harry Furstenberg.
\newblock Ergodic behavior of diagonal measures and a theorem of szemer�di on
  arithmetic progressions.
\newblock {\em Journal d'Analyse Math�matique}, 31:204--256, 1977.
\newblock 10.1007/BF02813304.

\bibitem{goldblatt}
R.~Goldblatt.
\newblock {\em Lectures on the Hyperreals: An introduction to nonstandard
  analysis}, volume 188 of {\em Graduate Texts in Mathematics}.
\newblock Springer-Verlag, 1998.

\bibitem{gowers01}
W.~T. Gowers.
\newblock A new proof of {S}zemer\'edi's theorem.
\newblock {\em Geometric And Functional Analysis}, 11(3):465--588, Aug 2001.
\newblock 10.1007/s00039-001-0332-9.

\bibitem{gowers:MR2373376}
W.~T. Gowers.
\newblock Hypergraph regularity and the multidimensional {S}zemer\'edi theorem.
\newblock {\em Ann. of Math. (2)}, 166(3):897--946, 2007.

\bibitem{green:MR2391635}
Ben Green and Terence Tao.
\newblock An inverse theorem for the {G}owers {$U\sp 3(G)$} norm.
\newblock {\em Proc. Edinb. Math. Soc. (2)}, 51(1):73--153, 2008.

\bibitem{green:MR2415379}
Ben Green and Terence Tao.
\newblock The primes contain arbitrarily long arithmetic progressions.
\newblock {\em Ann. of Math. (2)}, 167(2):481--547, 2008.

\bibitem{green:MR2747135}
Ben Green, Terence Tao, and Tamar Ziegler.
\newblock An inverse theorem for the {G}owers {$U^4$}-norm.
\newblock {\em Glasg. Math. J.}, 53(1):1--50, 2011.

\bibitem{host05}
B.~Host and B.~Kra.
\newblock Nonconventional ergodic averages and nilmanifolds.
\newblock {\em Ann. of Math. (2)}, 161(1):397--488, 2005.
\newblock 10.4007/annals.2005.161.397.

\bibitem{hrushovski}
Ehud Hrushovski.
\newblock Stable group theory and approximate subgroups.
\newblock http://www.math.huji.ac.il/~ehud/NQF/nqf.pdf, December 2009.

\bibitem{kechris}
A.~Kechris.
\newblock {\em Classical Descriptive Set Theory}, volume 156 of {\em Graduate
  Texts in Mathematics}.
\newblock Springer-Verlag, 1995.

\bibitem{keisler:MR819545}
H.~J. Keisler.
\newblock Probability quantifiers.
\newblock In {\em Model-theoretic logics}, Perspect. Math. Logic, pages
  509--556. Springer, New York, 1985.

\bibitem{kimpillay}
B.~Kim and A.~Pillay.
\newblock Simple theories.
\newblock {\em Ann. Pure Appl. Logic}, 88(2-3):149--164, 1997.

\bibitem{kohayakawa:MR1661982}
Y.~Kohayakawa.
\newblock Szemer\'edi's regularity lemma for sparse graphs.
\newblock In {\em Foundations of computational mathematics ({R}io de {J}aneiro,
  1997)}, pages 216--230. Springer, Berlin, 1997.

\bibitem{nagle:MR2198495}
Brendan Nagle, Vojtech R{\"o}dl, and Mathias Schacht.
\newblock The counting lemma for regular {$k$}-uniform hypergraphs.
\newblock {\em Random Structures Algorithms}, 28(2):113--179, 2006.

\bibitem{szemeredi:MR540024}
Endre Szemer{\'e}di.
\newblock Regular partitions of graphs.
\newblock In {\em Probl\`emes combinatoires et th\'eorie des graphes ({C}olloq.
  {I}nternat. {CNRS}, {U}niv. {O}rsay, {O}rsay, 1976)}, volume 260 of {\em
  Colloq. Internat. CNRS}, pages 399--401. CNRS, Paris, 1978.

\bibitem{taomeasure}
T.~Tao.
\newblock {\em An Introduction to Measure Theory}, volume 126 of {\em Graduate
  Studies in Mathematics}.
\newblock American Mathematical Society, 2011.

\bibitem{tao07}
Terence Tao.
\newblock A correspondence principle between (hyper)graph theory and
  probability theory, and the (hyper)graph removal lemma.
\newblock {\em Journal d'Analyse Math�matique}, 103:1--45, 2007.
\newblock 10.1007/s11854-008-0001-0.

\bibitem{tao08Norm}
Terence Tao.
\newblock Norm convergence of multiple ergodic averages for commuting
  transformations.
\newblock {\em Ergodic Theory Dynam. Systems}, 28(2):657--688, 2008.

\bibitem{towsner:MR2529651}
Henry Towsner.
\newblock Convergence of diagonal ergodic averages.
\newblock {\em Ergodic Theory Dynam. Systems}, 29(4):1309--1326, 2009.
\newblock 10.1017/S0143385708000722.

\bibitem{henry12:_analy_approac_spars_hyper}
Henry Towsner.
\newblock An analytic approach to sparse hypergraphs: Hypergraph removal.
\newblock submitted, 2012.

\end{thebibliography}

\end{document}